\documentclass{article}
\usepackage{amsmath, amssymb, amsfonts, bbm, verbatim, multirow}
\usepackage{psfrag}
\usepackage{graphicx}
\usepackage{amsthm}
\usepackage{fullpage}
\usepackage{natbib}
\usepackage{enumerate}
\usepackage{array}
\usepackage{rotating}
\usepackage[small]{caption}

\newtheorem{thm}{Theorem}
\newtheorem{defn}{Definition}
\newtheorem{lemma}[thm]{Lemma}

\newtheorem{prop}[thm]{Proposition}

\newcommand{\floor}[1]{\left\lfloor #1 \right\rfloor}

\DeclareMathOperator{\supp}{supp}
\DeclareMathOperator*{\argmax}{argmax}
\newenvironment{prooftitle}[1]{{\noindent \textsc{Proof #1}}\\}

\begin{document}

\title{Variable selection with error control: Another look at Stability Selection}
\author{Rajen Shah and Richard J. Samworth\\ 
  Statistical Laboratory\\
  University of Cambridge\\
  \{r.shah, r.samworth\}@statslab.cam.ac.uk
}



\def\scaleofComp_plots_9_4{0.7}
\def\scale_bounds_compare.ps{0.8}
\def\scaleof_Comp_plots_9_2_2{0.8}

\maketitle


\begin{abstract}
Stability Selection was recently introduced by \citet{MeinshausenBuhlmann2010} as a very general technique designed to improve the performance of a variable selection algorithm.  It is based on aggregating the results of applying a selection procedure to subsamples of the data.  We introduce a variant, called Complementary Pairs Stability Selection (CPSS), and derive bounds both on the expected number of variables included by CPSS that have low selection probability under the original procedure, and on the expected number of high selection probability variables that are excluded.  These results require no (e.g. exchangeability) assumptions on the underlying model or on the quality of the original selection procedure.  Under reasonable shape restrictions, the bounds can be further tightened, yielding improved error control, and therefore increasing the applicability of the methodology.    

\bigskip


\end{abstract}

\section{Introduction}

The problem of variable selection has received a huge amount of attention over the last 15 years, motivated by the desire to understand structure in massive data sets that are now routinely encountered across many scientific disciplines.  It is now very common, e.g. in biological applications, image analysis and portfolio allocation problems as well as many others, for the number of variables (or predictors) $p$ that are measured to exceed the number of observations $n$.  In such circumstances, variable selection is essential for model interpretation.

In a notable recent contribution to the now vast literature on this topic, \citet{MeinshausenBuhlmann2010} proposed Stability Selection as a very general technique designed to improve the performance of a variable selection algorithm.  The basic idea is that instead of applying one's favourite algorithm to the whole data set to determine the selected set of variables, one instead applies it several times to random subsamples of the data of size $\floor{n/2}$, and chooses those variables that are selected most frequently on the subsamples.  Stability Selection is therefore intimately connected with bagging \citep{Breiman1996,Breiman1999} and subagging \citep{BuhlmannYu2002}.  

A particularly attractive feature of Stability Selection is the error control provided by an upper bound on the expected number of falsely selected variables \citep[][Theorem~1]{MeinshausenBuhlmann2010}.  Such control is typically unavailable when applying the original selection procedure to the whole data set, and allows the practitioner to select the threshold $\tau$ for the proportion of subsamples for which a variable must be selected in order for it to be declared significant.

However, the bound does have a couple of drawbacks.  Firstly, it applies to the `population version' of the subsampling process, i.e. to the version of the procedure that aggregates results over the non-random choice of all $\binom{n}{\floor{n/2}}$ subsamples.  Even for $n$ as small as 15, it is unrealistic to expect this version to be used in practice, and in fact choosing around $100$ random subsamples is probably typical.  More seriously, the bound is derived under a very strong exchangeability assumption on the selection of noise variables (as well as a weak one on the quality of the original selection procedure, namely that it is not worse than random guessing).    

In this paper, we develop the methodology and conceptual understanding of Stability Selection in several respects.  We introduce a variant of Stability Selection, where the subsamples are drawn as complementary pairs from $\{1,\ldots,n\}$.  Thus the subsampling procedure outputs index sets $\{(A_{2j-1},A_{2j}):j = 1,\ldots,B\}$, where each $A_j$ is a subset of $\{1,\ldots,n\}$ of size $\floor{n/2}$, and $A_{2j-1} \cap A_{2j} = \emptyset$.  We call this variant Complementary Pairs Stability Selection (CPSS).  

At first glance it would seem that CPSS would be expected to yield very similar results to the original version of Stability Selection.  However, we show that CPSS in fact has the following properties:
\begin{enumerate}[(i)]
\item The Meinshausen--B\"uhlmann bound holds for CPSS regardless of the number of complementary pairs $B$ chosen -- even with $B = 1$.
\item There is a corresponding bound for the number of important variables excluded by CPSS.
\item Our results have no conditions on the original selection procedure, and in particular do not require the strong exchangeability assumption on the selection of noise variables.  Indeed, we argue that even a precise definition of `signal' and `noise' variables is not helpful in trying to understand the properties of CPSS, and we instead state the bounds in terms of the expected number of variables chosen by CPSS that have low selection probability under the base selection procedure, and the expected number of high selection probability variables that are excluded by CPSS.  See Section~\ref{Sec:CPSS} for further discussion.
\item The bound on the number of low selection probability variables chosen by CPSS can be significantly sharpened under mild shape restrictions (e.g. unimodality or $r$-concavity) on the distribution of the proportion of times a variable is selected in both $A_{2j-1}$ and $A_{2j}$.  We discuss these conditions in detail in Sections~\ref{sec:ImprovedTail} and~\ref{sec:r_concave} respectively, and compare both the original and new bounds to demonstrate the marked improvement.
\end{enumerate}
Our improved bounds are based on new versions of Markov's inequality that hold for random variables whose distributions are unimodal or $r$-concave.  However, it is important to note at this point that the results are not just a theoretical contribution; they allow the practitioner to reduce $\tau$ (and therefore select more variables) for the same control of the number of low selection probability variables chosen by CPSS.  In Section~\ref{Sec:Practice}, we give recommendations on how a practitioner can make use of the bounds in applying CPSS. 


In Section~\ref{Sec:SimStudy}, we present the results of an extensive simulation study designed to illustrate the appropriateness of our shape restrictions, and to compare Stability Selection and CPSS with their base selection procedures. 

A review of some of the extensive literature on variable selection can be found in~\citet{FanLv2010}.  Work related more specifically to Stability Selection includes \citet{Bach2008}, who studied the Bolasso (short for Bootstrapped enhanced Lasso).  This involves applying the Lasso to bootstrap (with replacement) samples from the original data, rather than subsampling without replacement.  A final estimate is obtained by applying the Lasso to the intersection of the set of variables selected across the bootstrap samples.  Various authors, particularly in the machine learning literature, have considered the \emph{stability} of a feature selection algorithm, i.e. the insensitivity of the output of the algorithm to variations in the training set; such studies include \citet{LBRB2003}, \citet{KPH2007}, \citet{Kuncheva2007}, \citet{LYD2009} and \citet{HanYu2010}.  \citet{SAP2008} consider obtaining a final feature ranking by aggregating the rankings across bootstrap samples.

\section{Complementary Pairs Stability Selection}
\label{Sec:CPSS}

In order to keep our discussion rather general, we only assume that we have vector-valued data $z_1,\ldots,z_n$ which we take to be a realisation of independent and identically distributed random elements $Z_1,\ldots,Z_n$.  Informally, we think of some of the components of $Z_i$ as being `signal variables', and others as being `noise variables', though for our purposes it is not necessary to define these notions precisely.  Formally, we let $S \subseteq \{1,\ldots,p\}$ and $N := \{1,\ldots,p\} \setminus S$, thought of as the index sets of the signal and noise variables respectively.  A \emph{variable selection procedure} is a statistic $\hat{S}_n := \hat{S}_n(Z_1,\ldots,Z_n)$ taking values in the set of all subsets of $\{1,\ldots,p\}$, and we think of $\hat{S}_n$ as an estimator of $S$.  As a typical example, we may often write $Z_i = (X_i, Y_i)$ with the covariate $X_i \in \mathbb{R}^p$ and the response $Y_i \in \mathbb{R}$, and our (pseudo) log-likelihood might be of the form
\begin{equation}
\label{Eq:LogLike}
\sum_{i=1}^n L(Y_i,X_i^T\beta),
\end{equation}
for some $\beta \in \mathbb{R}^p$.  In this context, we regard $S:= \{k:\beta_k \neq 0\}$ as the signal indices, $N = \{k:\beta_k=0\}$ as noise indices.  Examples from graphical modelling can also be cast within our framework.  Note however that we do not require a (pseudo) log-likelihood of the form~(\ref{Eq:LogLike}). 

We define the selection probability of a variable index $k \in \{1,\ldots,p\}$ under $\hat{S}_n$ as
\begin{equation}
p_{k,n} = \mathbb{P}(k \in \hat{S}_n) = \mathbb{E}(\mathbbm{1}_{\{k \in \hat{S}_n\}}).
\end{equation}
We take the view that for understanding the properties of Stability Selection, the selection probabilities $p_{k,n}$ are the fundamental quantities of interest.  Since an application of Stability Selection is contingent on a choice of base selection procedure $\hat{S}_n$, all we can hope is that it selects variables having high selection probability under the base procedure, and avoids selecting those variables with low selection probability.  Indeed this turns out to be the case; see Theorem~\ref{thm:hi_low} below.

Of course, $\mathbbm{1}_{\{k \in \hat{S}_n\}}$ has a Bernoulli distribution with parameter $p_{k,n}$, so we may view $\mathbbm{1}_{\{k \in \hat{S}_n\}}$ as an unbiased estimator of $p_{k,n}$ (though $p_{k,n}$ is not a model parameter in the conventional sense).  The key idea of Stability Selection is to improve on this simple estimator of $p_{k,n}$ through subsampling.


For a subset $A = \{i_1, \ldots, i_{|A|}\} \subset \{1,\ldots,n\}$ with $i_i < \cdots < i_{|A|}$, we shall write
\[
 \hat{S}(A) := \hat{S}_{|A|}(Z_{i_1}, \ldots, Z_{i_{|A|}}).
\]
\begin{defn}[Complementary Pairs Stability Selection]
Let $\{(A_{2j-1},A_{2j}):j = 1,\ldots,B\}$ be randomly chosen independent pairs of subsets of $\{1,\ldots,n\}$ of size $\floor{n/2}$ such that $A_{2j-1} \cap A_{2j} = \emptyset$.  For $\tau \in [0,1]$, the Complementary Pairs Stability Selection version of a variable selection procedure $\hat{S}_n$ is $\hat{S}_{n,\tau}^{\mathrm{CPSS}} = \{k:\hat{\Pi}_B(k) \geq \tau\}$, where the function $\hat{\Pi}_B : \{1,\ldots,p\} \to \{0, \tfrac{1}{2B},\tfrac{1}{B}, \ldots, 1\}$ is given by
\begin{equation} 
\label{eq:StabSelect}
\hat{\Pi}_B(k) := \frac{1}{2B} \sum_{j=1}^{2B} \mathbbm{1}_{\{k \in \hat{S}(A_j)\}}.
\end{equation}
\end{defn}
Note that $\hat{\Pi}_B(k)$ is an unbiased estimator of $p_{k,\floor{n/2}}$, but, in general, a biased estimator of $p_{k,n}$. However, by means of the averaging involved in \eqref{eq:StabSelect}, we hope that $\hat{\Pi}_B(k)$ will have reduced variance compared with $\mathbbm{1}_{\{k \in \hat{S}_n\}}$, and that this increased stability will more than compensate for the bias incurred.  Indeed, this is the case in other situations where bagging and subagging have been successfully applied, such as classification trees \citep{Breiman1996} or nearest neighbour classifiers \citep{HallSamworth2005,BCG2010,Samworth2011}.  


An alternative to subsampling complementary pairs would be to use bootstrap sampling.  We have found that this gives very similar estimates of $p_{k,n}$, though most of our theoretical arguments do not apply when the bootstrap is used (the approach in Section~\ref{Sec:lowertau} is an exception in this regard).  In fact, taking subsamples of size $\floor{n/2}$ can be thought of as the subsampling scheme that most closely mimics the bootstrap \citep[e.g.][]{DSS2011}.

It is convenient at this stage to define another related selection procedure based on sample splitting.
\begin{defn}[Simultaneous Selection]
Let $\{(A_{2j-1},A_{2j}):j = 1,\ldots,B\}$ be randomly chosen independent pairs of subsets of $\{1,\ldots,n\}$ of size $\floor{n/2}$ such that $A_{2j-1} \cap A_{2j} = \emptyset$.  For $\tau \in [0,1]$, the Simultaneous Selection version of $\hat{S}_n$ is $\hat{S}_{n,\tau}^{\mathrm{SIM}} = \{k:\tilde{\Pi}_B(k) \geq \tau\}$, where 
\begin{equation}
\label{Eq:simult}
\tilde{\Pi}_B(k) := \frac{1}{B} \sum_{j=1}^B \mathbbm{1}_{\{k \in \hat{S}(A_{2j-1})\}}\mathbbm{1}_{\{k \in \hat{S}(A_{2j})\}}.
\end{equation}
\end{defn}
For our purposes, Simultaneous Selection is a tool for understanding the properties of CPSS.  However, the special case of $B=1$ of Simultaneous Selection was studied by \citet{FSW2009}, and a variant involving all possible disjoint pairs of subsets was considered in \citet{MeinshausenBuhlmann2010}.

\section{Theoretical properties}

\subsection{Worst-case bounds}
\label{Sec:WorstCase}

In Theorem~\ref{thm:hi_low} below, we show that the expected number of low selection probability variables chosen by CPSS is controlled in terms of the expected number chosen by the original selection procedure, with a corresponding result for the expected number of high selection probability variables not chosen by CPSS.  The appealing feature of these results is their generality: they require no assumptions on the underlying model or on the quality of the original selection procedure, and they apply regardless of the number $B$ of complementary pairs of subsets chosen.  

For $\theta \in [0,1]$, let $L_\theta = \{k: p_{k,\floor{n/2}} \leq \theta\}$ denote the set of variable indices that have low selection probability under $\hat{S}_{\floor{n/2}}$, and let $H_{\theta} = \{k:p_{k,\floor{n/2}} > \theta\}$ denote the set of those that have high selection probability.
\begin{thm}
\label{thm:hi_low}
\begin{enumerate}[(i)]
\item If $\tau \in (\frac{1}{2},1]$, then
\[
\mathbb{E}|\hat{S}_{n,\tau}^{\mathrm{CPSS}} \cap L_\theta| \leq \frac{\theta}{2\tau-1} \mathbb{E}|\hat{S}_{\floor{n/2}} \cap L_\theta|.
\]
\item Let $\hat{N}_{n,\tau}^{\mathrm{CPSS}} = \{1,\ldots,p\} \setminus \hat{S}_{n,\tau}^{\mathrm{CPSS}}$ and $\hat{N}_n = \{1,\ldots,p\} \setminus \hat{S}_n$.  If $\tau \in [0,\frac{1}{2})$, then
\[
\mathbb{E}|\hat{N}_{n,\tau}^{\mathrm{CPSS}} \cap H_\theta| \leq \frac{1-\theta}{1-2\tau} \mathbb{E}|\hat{N}_{\floor{n/2}} \cap H_\theta|.
\]
\end{enumerate}
\end{thm}
In many applications, and for a good base selection procedure, we imagine that the set of selection probabilities $\{p_{k,\floor{n/2}}:k=1,\ldots,p\}$ is positively skewed in $[0,1]$, with many selection probabilities being very low (predominantly noise variables), and with just a few being large (including at least some of the signal variables).
To illustrate Theorem~\ref{thm:hi_low}(i), consider a situation with $p=1000$ variables and where the base selection procedure chooses 50 of them.  Then Theorem~\ref{thm:hi_low}(i) shows that on average CPSS with $\tau = 0.6$ selects no more than a quarter of the below average selection probability variables chosen by $\hat{S}_{\floor{n/2}}$.

Our Theorem~\ref{thm:hi_low}(i) is analogous to Theorem~1 of \citet{MeinshausenBuhlmann2010}.  The differences are that we do not require the condition that $\{\mathbbm{1}_{\{k \in \hat{S}_{\floor{n/2}}\}}:k \in N\}$ is exchangeable, nor that the original procedure is no worse than random guessing, and our result holds for all $B$.  The price we pay is that the bound is stated in terms of the expected number of low selection probability variables chosen by CPSS, rather than the expected number of noise variables, which we do for the reasons described in Section~\ref{Sec:CPSS}.  If the exchangeability and random guessing conditions mentioned above do hold, then, writing $q := \mathbb{E}|\hat{S}_{\floor{n/2}}|$, we recover
\[
\mathbb{E}|\hat{S}_{n,\tau}^{\mathrm{CPSS}} \cap N| \leq \frac{1}{2\tau-1} \Bigl(\frac{q}{p}\Bigr)\mathbb{E}|\hat{S}_{\floor{n/2}} \cap L_{q/p}| \leq \frac{1}{2\tau-1} \Bigl(\frac{q^2}{p}\Bigr).
\]
The final bound here was obtained in Theorem~1 of \citet{MeinshausenBuhlmann2010} for the population version of Stability Selection.

\subsection{Improved bounds under unimodality} \label{sec:ImprovedTail}

Despite the attractions of Theorem~\ref{thm:hi_low}, the following observations suggest there may be scope for improvement.  Firstly, we expect we should be able to obtain tighter bounds as $B$ increases.  Secondly, and more importantly, examination of the proof of Theorem~\ref{thm:hi_low}(i) shows that our bound relies on first noting that
\begin{equation} \label{eq:simult_stab}
 1 + \tilde{\Pi}_B(k) \geq 2 \hat{\Pi}_B(k),
\end{equation}
and then applying Markov's inequality to $\tilde{\Pi}_B(k)$.
For equality in Markov's inequality, $\tilde{\Pi}_B(k)$ must be a mixture of point masses at $0$ and $2\tau-1$, but Figure~\ref{fig:compare_dist} suggests that the distribution of $\tilde{\Pi}_B(k)$, which is supported on $\{0,\frac{1}{B},\frac{2}{B},\ldots,1\}$, can be very different from this.  Indeed, our experience, based on extensive simulation studies, is that when $\theta$ is close to $q/p$ (which is where the bound in Theorem~\ref{thm:hi_low}(i) is probably of most interest), the distribution of $\tilde{\Pi}_B(k)$ over $k \in L_\theta$ is remarkably consistent over different data generating processes, and Figure~\ref{fig:compare_dist} is typical.
\begin{figure}[p]
	\begin{center}
		\psfragscanon
		\psfrag{0.0}[][][\scaleofComp_plots_9_4]{0.0}
		\psfrag{0.2}[][][\scaleofComp_plots_9_4]{0.2}
		\psfrag{0.4}[][][\scaleofComp_plots_9_4]{0.4}
		\psfrag{0.6}[][][\scaleofComp_plots_9_4]{0.6}
		\psfrag{0.8}[][][\scaleofComp_plots_9_4]{0.8}
		\psfrag{1.0}[][][\scaleofComp_plots_9_4]{1.0}
		\psfrag{0}[][][\scaleofComp_plots_9_4]{0}
		\psfrag{2}[][][\scaleofComp_plots_9_4]{2}
		\psfrag{4}[][][\scaleofComp_plots_9_4]{4}
		\psfrag{6}[][][\scaleofComp_plots_9_4]{6}
		\psfrag{8}[][][\scaleofComp_plots_9_4]{8}
		\psfrag{10}[][][\scaleofComp_plots_9_4]{10}
		\psfrag{1.2}[][][\scaleofComp_plots_9_4]{1.2}
		\psfrag{0}[][][\scaleofComp_plots_9_4]{0}
		\psfrag{50}[][][\scaleofComp_plots_9_4]{50}
		\psfrag{100}[][][\scaleofComp_plots_9_4]{100}
		\psfrag{150}[][][\scaleofComp_plots_9_4]{150}
		\psfrag{200}[][][\scaleofComp_plots_9_4]{200}
		\psfrag{250}[][][\scaleofComp_plots_9_4]{250}
		\psfrag{Probability}[][][0.8]{Probability}
		\psfrag{Probability*10^2}[][][0.8]{$\mathrm{Probability} \times 10^2$}
		\psfrag{Probability^(-1/2)}[][][0.8]{$\mathrm{Probability}^{-1/2}$}
		\psfrag{a}[][][0.7]{Worst case}
		\psfrag{b}[][][0.7]{Unimodal}
		\psfrag{c}[][][0.7]{$-\frac{1}{2}$-concave}
		\psfrag{d}[][][0.7]{Empirical}
		\includegraphics[width=\textwidth]{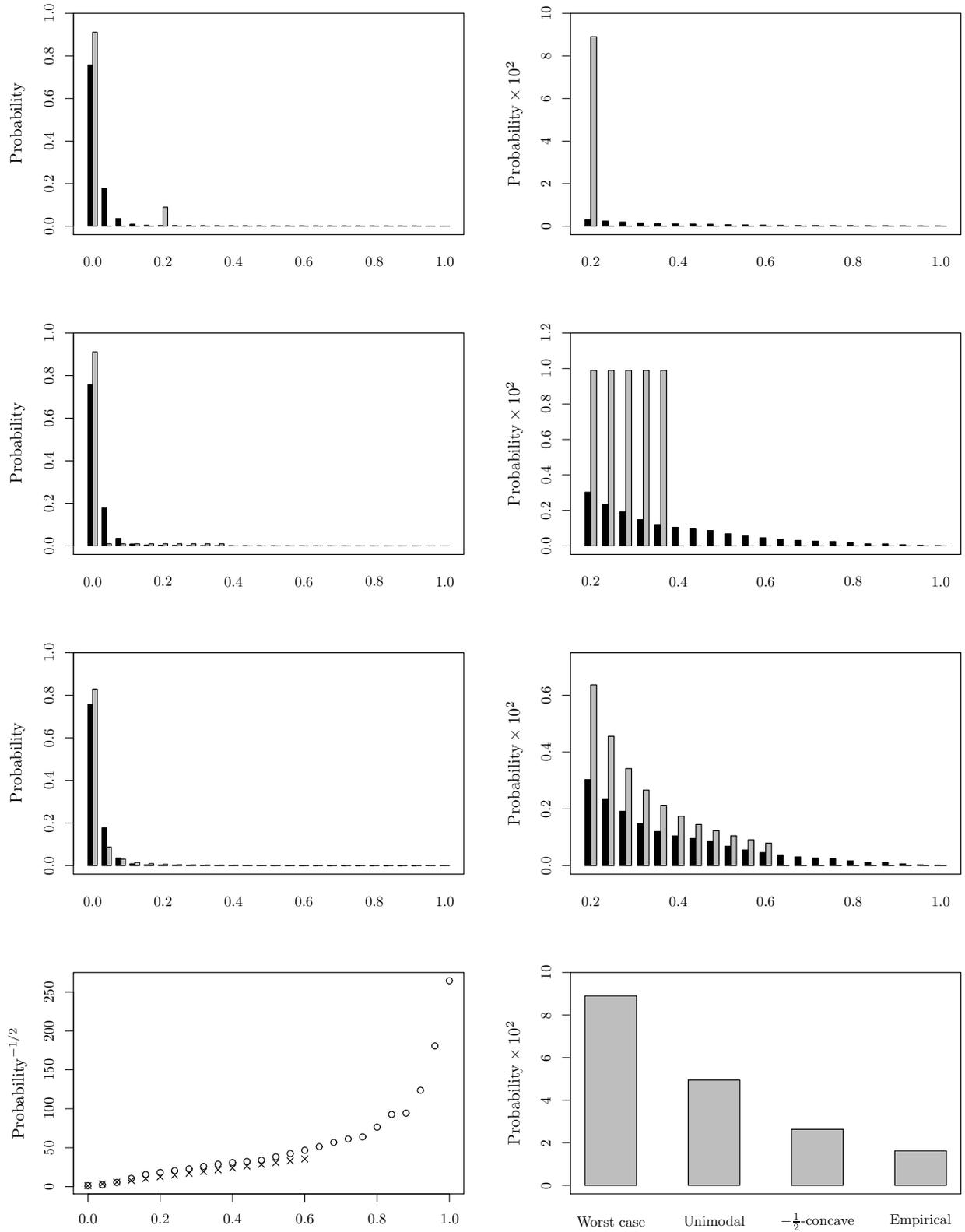}
    \caption{Rows 1 to 3 show a typical example of the full probability mass function (left) and zoomed in from 0.2 onwards (right) of $\tilde{\Pi}_{25}(k)$ for $k \in L_{q/p}$ (black), alongside the unrestricted, unimodal and $-1/2$-concave distributions respectively (grey), which have maximum tail probability beyond 0.2.  This situation corresponds to selecting $\tau = 0.6$.  Bottom left: the observed mass function (circles) and the extremal $-1/2$-concave mass function (crosses) on the $x^{-1/2}$ scale.  Bottom right: tail probabilities from 0.2 onwards for each of the distributions.}   \label{fig:compare_dist}
	\end{center}
\end{figure}
It is therefore natural to consider placing shape restrictions on the distribution of $\tilde{\Pi}_B(k)$ which encompass what we see in practice, and which yield stronger versions of Markov's inequality.  As a first step in this direction, we consider the assumption of unimodality.
\begin{thm}
\label{thm:Unimodal}
Suppose that the distribution of $\tilde{\Pi}_B(k)$ is unimodal for each $k \in L_\theta$.  If $\tau \in \{\frac{1}{2} + \frac{1}{B},\frac{1}{2}+\frac{3}{2B},\frac{1}{2}+\frac{2}{B},\ldots,1\}$, then
\[
\mathbb{E}|\hat{S}_{n,\tau}^{\mathrm{CPSS}} \cap L_\theta| \leq C(\tau,B)\, \theta \, \mathbb{E}|\hat{S}_{\floor{n/2}} \cap L_\theta|,
\]
where, when $\theta \leq 1/\sqrt{3}$,
\[
C(\tau,B) = \left\{ \begin{array}{ll} \vspace*{0.1in}
\dfrac{1}{2(2\tau - 1 - 1/2B)} & \mbox{if $\tau \in (\min(\frac{1}{2} + \theta^2,\frac{1}{2}+\frac{1}{2B} + \frac{3}{4}\theta^2),\frac{3}{4}]$} \\ 
\dfrac{4(1- \tau + 1/2B)}{1 + 1/B} & \mbox{if $\tau \in (\frac{3}{4},1]$}. \end{array} \right.
\]
\end{thm}
The proof of Theorem~\ref{thm:Unimodal} is based on a new version of Markov's inequality (Theorem~\ref{thm:UniMarkov} in the Appendix) for random variables with unimodal distributions supported on a finite lattice.  There is also an explicit expression for $C(\tau,B)$ when $\theta > 1/\sqrt{3}$, which follows from Theorem~\ref{thm:UniMarkov} in the same way, but we do not present it here because it is a little more complicated, and because we anticipate the bound when $\theta$ is (much) smaller than $1/\sqrt{3}$ being of most use in practice.  See Section~\ref{Sec:Practice} for further discussion.

Figure~\ref{fig:bounds_compare} compares the bounds provided by Theorems~\ref{thm:hi_low} and Theorem~\ref{thm:Unimodal} as a function of $\tau$, for the illustration discussed after the statement of Theorem~\ref{thm:hi_low}.  
\begin{figure}[htb]
	\begin{center}
		\psfragscanon
		\psfrag{0.4}[][][\scale_bounds_compare.ps]{0.4}
		\psfrag{0.5}[][][\scale_bounds_compare.ps]{0.5}
		\psfrag{0.6}[][][\scale_bounds_compare.ps]{0.6}
		\psfrag{0.7}[][][\scale_bounds_compare.ps]{0.7}
		\psfrag{0.8}[][][\scale_bounds_compare.ps]{0.8}
		\psfrag{0.9}[][][\scale_bounds_compare.ps]{0.9}
		\psfrag{1.0}[][][\scale_bounds_compare.ps]{1.0}
		\psfrag{0}[][][\scale_bounds_compare.ps]{0}
		\psfrag{5}[][][\scale_bounds_compare.ps]{5}
		\psfrag{10}[][][\scale_bounds_compare.ps]{10}
		\psfrag{15}[][][\scale_bounds_compare.ps]{15}
  		\psfrag{Error bounds}[][][0.9]{Bound on $\mathbb{E}|\hat{S}_{n,\tau}^{\mathrm{CPSS}} \cap L_{q/p}|$}
		\psfrag{tau}[][][0.9]{$\tau$}
		\includegraphics[height = \textwidth, angle = 270, clip = true, trim = 55 0 18 10]{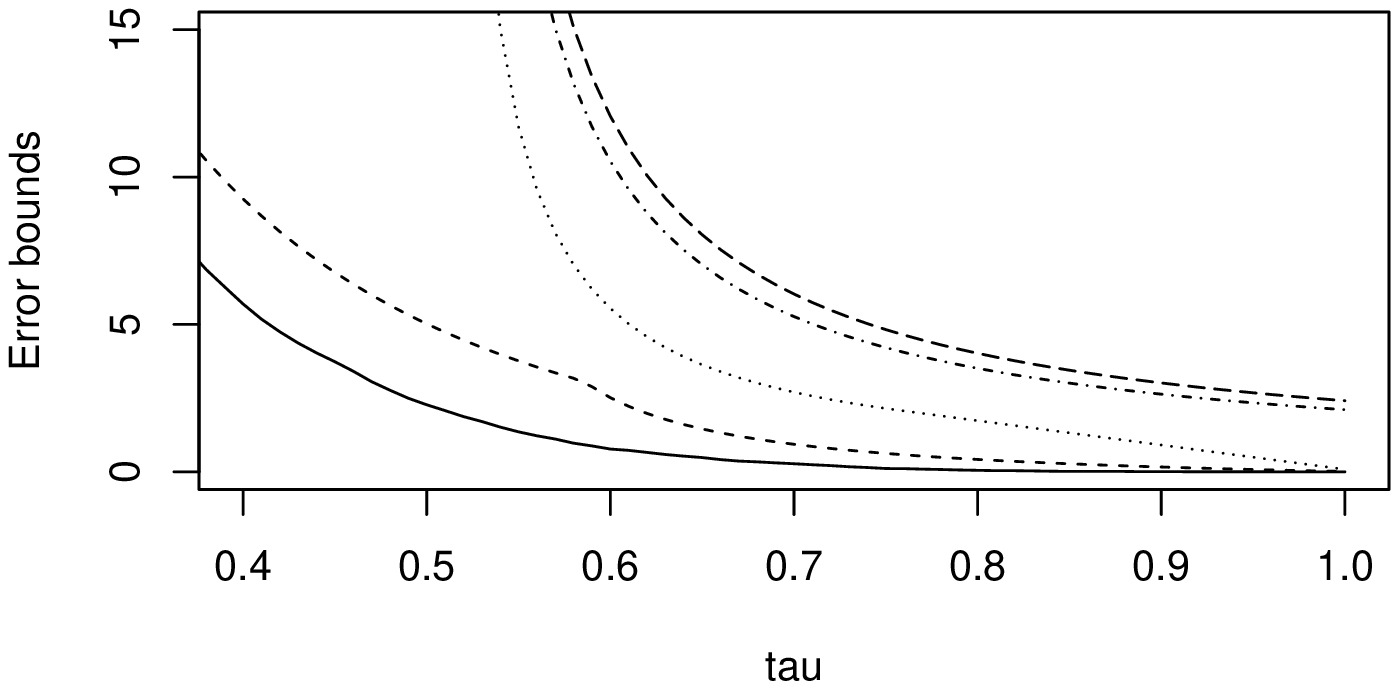}
  \caption{Comparison of the bounds on $\mathbb{E}|\hat{S}_{n,\tau}^{\mathrm{CPSS}} \cap L_{q/p}|$ for different values of the threshold $\tau$: the original bound from Theorem 1 of \citet{MeinshausenBuhlmann2010} (long dashes), our worst case bound (dots and dashes), the unimodal bound (dots) and the $r$-concave bound \eqref{eq:FinalBound} (short dashes). The solid line is the true value of $\mathbb{E}|\hat{S}_{n,\tau}^{\mathrm{CPSS}} \cap L_{q/p}|$ for a simulated example. In this case $p=1000$, $q=50$ and the number of signal variables was 8.
}
  \label{fig:bounds_compare}
	\end{center}
\end{figure}

\subsection{Further improvements under $r$-concavity} 
\label{sec:r_concave}

The unimodal assumption allows for a significant improvement in the bounds attainable from a naive application of Markov's inequality.  However, Figure~\ref{fig:compare_dist} suggests that further gains may be realised by placing tighter constraints on the family of distributions for $\tilde{\Pi}_B(k)$ that we consider, in order to match better the empirical distributions that we see in practice. 

A very natural constraint to impose on the distribution of $\tilde{\Pi}_B(k)$ is log-concavity.  By this, we mean that, if $f$ denotes the probability mass function of $\tilde{\Pi}_B(k)$, then the linear interpolant to $\{(i,f(i/B)):i=0,1,\ldots,B\}$ is a log-concave function on $[0,1]$.  Log-concavity is a shape constraint that has received a great deal of attention recently (e.g. \citet{Walther2002,DuembgenRufibach2009,CSS2010}), and at first sight it seems reasonable in our context, because if the summands in \eqref{Eq:simult} were independent, then we would have $\tilde{\Pi}_B(k) \sim \frac{1}{B}\mathrm{Bin}(B,p_{k,\floor{n/2}}^2)$, which is log-concave.

It is indeed possible to obtain a version of Markov's inequality under log-concavity that leads to another improvement in the bound on $\mathbb{E}|\hat{S}_{n,\tau}^{\mathrm{CPSS}} \cap L_\theta|$.  However, we found that in practice, the dependence structure of the summands in \eqref{Eq:simult} meant that the log-concavity constraint was a little too strong.  We therefore consider instead the class of \emph{$r$-concave} distributions, which we claim defines a continuum of constraints that interpolate between log-concavity and unimodality (see Propositions~\ref{prop:log_iff._r} and~\ref{prop:uni_implies_r} below).  This constraint has also been studied recently in the context of density estimation by \citet{SereginWellner2010} and \citet{KoenkerMizera2010}; see also \citet{DJ1988}.

To define the class, we recall that the $r^{\mathrm{th}}$ generalised mean $M_r (a, b; \lambda)$ of $a,b \geq 0$ is given by
\[
 M_r (a, b; \lambda) = \{ (1-\lambda)a^r + \lambda b^r\}^{1/r}
\]
for $r>0$. This is also well-defined for $r<0$ if we take $ M_r (a, b; \lambda) =0$ when $ab=0$, and define $0^r = \infty$. In addition, we may define
\begin{gather*}
 M_0 (a,b ; \lambda) := \lim_{r \to 0} M_r (a, b; \lambda) = a^{1-\lambda}b^\lambda \\
M_{-\infty} (a,b ; \lambda) := \lim_{r \to -\infty} M_r (a, b; \lambda) = \min(a, b).
\end{gather*}
We are now in a position to define $r$-concavity.
\begin{defn}
 A non-negative function $f$ on an interval $I \subset \mathbb{R}$ is \emph{$r$-concave} if for every $x, y \in I$ and $\lambda \in (0,1)$, we have
\[
f((1-\lambda)x +  \lambda y) \geq M_r (f(x), f(y); \lambda).
\]
\end{defn}
\begin{defn}
 A probability mass function $f$ supported on $\{0,\frac{1}{B},\frac{2}{B},\ldots,1\}$ is \emph{$r$-concave} if the linear interpolant to $\{(i,f(i/B)):i=0,1,\ldots,B\}$ is $r$-concave.
\end{defn}
When $r < 0$, it is easy to see that $f$ is $r$-concave if and only if $f^r$ is convex.  Let $\mathcal{F}_r$ denote the class of $r$-concave probability mass functions on $\{0,\frac{1}{B},\frac{2}{B},\ldots,1\}$.  Then each $f \in \mathcal{F}_r$ is unimodal, and as $M_r (a, b; \lambda)$ is non-decreasing in $r$ for fixed $a$ and $b$, we have $\mathcal{F}_r \supset \mathcal{F}_{r'}$ for $r < r'$. Furthermore, $f$ is unimodal if it is $-\infty$-concave, and $f$ is log-concave if it is $0$-concave. The following two results further support the interpretation of $r$-concavity for $r \in [-\infty,0]$ as an interpolation between log-concavity and unimodality.
\begin{prop} \label{prop:log_iff._r}
 A function $f$ is log-concave if and only if it is $r$-concave for every $r<0$.
\end{prop}
\begin{prop} \label{prop:uni_implies_r}
 Let $f$ be a unimodal probability mass function supported on $\{0,\frac{1}{B},\frac{2}{B},\ldots,1\}$ and suppose both that $f(0) < \ldots < f(\frac{l}{B}) = f(\frac{l+1}{B}) = \ldots = f(\frac{u}{B})$ and that $f(\frac{u}{B}) > f(\frac{u+1}{B}) > \ldots > f(1)$, for some $l \leq u$. Then $f$ is $r$-concave for some $r < 0$.
\end{prop}
In Proposition~\ref{prop:rbound} in the Appendix, we present a result that characterises those $r$-concave distributions that attain equality in a version of Markov's inequality for random variables with $r$-concave distributions on $\{0,\frac{1}{B},\frac{2}{B},\ldots,1\}$.  If we assume that $\tilde{\Pi}_B(k)$ is $r$-concave for all $k \in L_\theta$, using \eqref{eq:simult_stab}, for these variables we can obtain a bound of the form
\begin{equation} \label{eq:D_ineq1}
 \mathbb{P}(\hat{\Pi}_B(k) \geq \tau) \leq D(p_{k, \floor{n/2}}^2, 2\tau - 1, B, r) \leq D(\theta^2, 2\tau - 1, B, r)
\end{equation}
where $D(\eta, t, B, r)$ denotes the maximum of $\mathbb{P}(X \geq t)$ over all $r$-concave random variables supported on $\{0,\frac{1}{B},\frac{2}{B},\ldots,1\}$ with $\mathbb{E}(X) \leq \eta$. Although $D$ does not appear to have a closed form, it is straightforward to compute numerically, as we describe in Section~\ref{Sec:Alg}.  The lack of a simple form means a direct analogue Theorem~\ref{thm:Unimodal} is not available. We can nevertheless obtain the following bound on the expected number of low selection probability variables chosen by CPSS:
\begin{equation} \label{eq:D_ineq2}
 \mathbb{E}|\hat{S}_{n,\tau}^{\mathrm{CPSS}} \cap L_\theta| = \sum_{k \in L_\theta} \mathbb{P}(\hat{\Pi}_B(k) \geq \tau) \leq D(\theta^2, 2\tau - 1, B, r) |L_\theta|.
\end{equation}
Our simulation studies suggest that $r=-1/2$ is a sensible choice to use for the bound.  In other words, if $f$ denotes the probability mass function of $\tilde{\Pi}_B(k)$, then the linear interpolant to $\{(i,f(i/B)^{-1/2}):i=0,1,\ldots,B\}$ is typically well approximated by a convex function.  This is illustrated in the bottom left panel of Figure~\ref{fig:compare_dist} (note that the right-hand tail in this plot corresponds to tiny probabilities).

\subsubsection{Lowering the threshold $\tau$}
\label{Sec:lowertau}

The bounds obtained thus far have used the relationship \eqref{eq:simult_stab} to convert a Markov bound for $\tilde{\Pi}_B(k)$ into a corresponding one for the statistic of interest, $\hat{\Pi}_B(k)$. The advantage of this approach is that $\mathbb{E}(\tilde{\Pi}_B(k)) = p_{k, \floor{n/2}} ^2$ is much smaller than $\mathbb{E}(\hat{\Pi}_B(k)) = p_{k, \floor{n/2}}$ for variables with low selection probability, so the Markov bound is quite tight. However, for $\tau$ close to $1/2$, the inequality \eqref{eq:simult_stab} starts to become weak, and bounds can only be obtained for $\tau > 1/2$ in any case.

To solve this problem, we can apply our versions of Markov's inequality directly to $\hat{\Pi}_B(k)$.  We have found, through our simulations, that for variables with low selection probability, the distribution of $\hat{\Pi}_B(k)$ can be modelled very well as a $-1/4$-concave distribution (see Figure~\ref{fig:minus_quarter_concave}). That the distribution of $\hat{\Pi}_B(k)$ is closer to log-concavity than that of $\tilde{\Pi}_B(k)$ is intuitive because although the summands in \eqref{eq:StabSelect} are not independent, terms involving subsamples which have little overlap will be close to independent. If we assume that $\tilde{\Pi}_B(k)$ is $-1/2$-concave and that $\hat{\Pi}_B(k)$ is $-1/4$-concave for all $k \in L_\theta$, we can obtain our best bound
\begin{equation} \label{eq:FinalBound}
  \mathbb{E}|\hat{S}_{n,\tau}^{\mathrm{CPSS}} \cap L_\theta| \leq \min\{D(\theta^2, 2\tau - 1, B, -1/4) , D(\theta, \tau, 2B, -1/2)\} |L_\theta|,
\end{equation}
which is valid for all $\tau \in (\theta, 1]$, provided we adopt the convention that $D(\cdot, t, \cdot, \cdot)=1$ for $t \leq 0$.  The resulting improvements in the bounds can been seen in Figure~\ref{fig:bounds_compare}.  Note the kink in Figure~\ref{fig:bounds_compare} for the $r$-concave bound \eqref{eq:FinalBound} just before $\tau = 0.6$.  This corresponds to the transition from where $D(\theta, \tau, 2B, -1/4)$ is smaller to where $D(\theta^2, 2\tau - 1, B, -1/2)$ is smaller.

We applied the algorithm described in Section~\ref{Sec:Alg} to produce tables of values of 
\[
 \min\{D(\theta^2, 2\tau - 1, 50, -1/2), D(\theta, \tau, 100, -1/4) \}
\]
over a grid of $\theta$ and $\tau$ values; see Table~\ref{tab:r_conc1} and Table~\ref{tab:r_conc2}.

\begin{figure}[htb]
	\begin{center}
		\psfragscanon
		\psfrag{0.0}[][][\scaleof_Comp_plots_9_2_2]{0.0}
		\psfrag{0.2}[][][\scaleof_Comp_plots_9_2_2]{0.2}
		\psfrag{0.4}[][][\scaleof_Comp_plots_9_2_2]{0.4}
		\psfrag{0.6}[][][\scaleof_Comp_plots_9_2_2]{0.6}
		\psfrag{0.8}[][][\scaleof_Comp_plots_9_2_2]{0.8}
		\psfrag{1.0}[][][\scaleof_Comp_plots_9_2_2]{1.0}
		\psfrag{0.1}[][][\scaleof_Comp_plots_9_2_2]{0.1}
		\psfrag{0.3}[][][\scaleof_Comp_plots_9_2_2]{0.3}
		\psfrag{0.5}[][][\scaleof_Comp_plots_9_2_2]{0.5}
		\psfrag{5}[][][\scaleof_Comp_plots_9_2_2]{5}
		\psfrag{10}[][][\scaleof_Comp_plots_9_2_2]{10}
		\psfrag{15}[][][\scaleof_Comp_plots_9_2_2]{15}
		\psfrag{20}[][][\scaleof_Comp_plots_9_2_2]{20}
		\psfrag{25}[][][\scaleof_Comp_plots_9_2_2]{25}
		\psfrag{Probability}[][][0.8]{Probability}
		\psfrag{Probability^(-1/4)}[][][0.8]{$\mathrm{Probability}^{-1/4}$}
		\includegraphics[width = \textwidth]{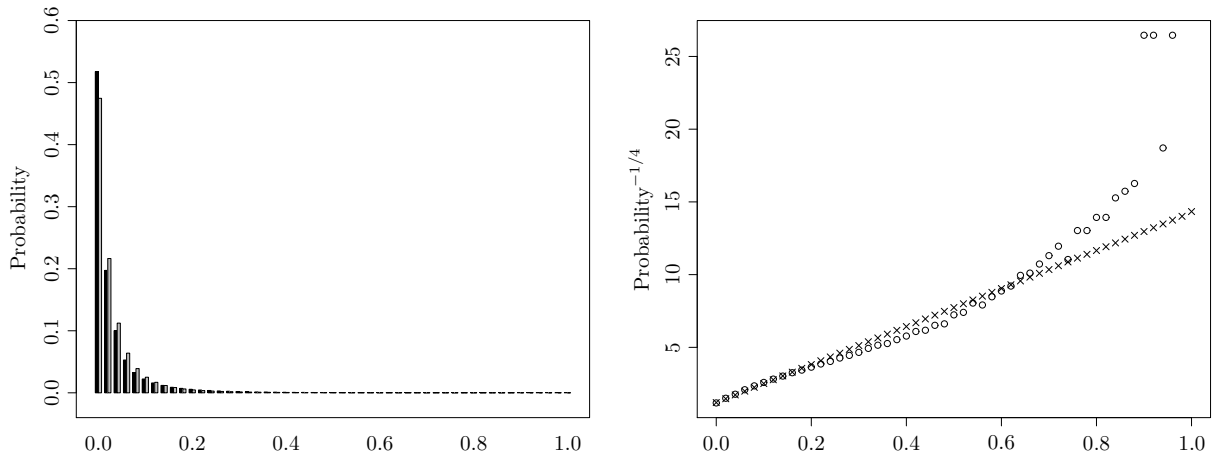}
  \caption{A typical example of the probability mass function of $\hat{\Pi}_{25}(k)$ for $k \in L_{q/p}$ (black bars and circles), alongside the $-1/4$-concave distribution (grey bars and crosses), which has maximum tail probability beyond 0.4.
}
  \label{fig:minus_quarter_concave}
	\end{center}
\end{figure}
\subsection{How to use these bounds in practice}
\label{Sec:Practice}
The quantities $|L_\theta|$ and $\mathbb{E}|\hat{S}_{\floor{n/2}} \cap L_\theta|$, which appear on the right hand sides of the bounds, will in general be unknown to the statistician. Thus when using the bounds, they will typically need to be replaced by $p$ and $q$ respectively.  In addition, several parameters must be selected, and in this section we go through each of these in turn and give guidance on how to choose them. 

\paragraph*{Choice of $B$.} We recommend $B=50$ as a default value.  Choosing $B$ larger than this increases the computational burden, and may lead to the $r$-concavity assumptions being violated.
\paragraph*{Choice of $\theta$.} As mentioned at the beginning of Section~\ref{sec:ImprovedTail}, $\theta = q/p$ is a natural choice. In other words, we regard the below average selection probability variables as the irrelevant variables.  Other choices of $\theta$ are possible, but the use of \eqref{eq:D_ineq1} and \eqref{eq:D_ineq2} to construct the bound suggests that the inequality will be tightest when most of the variables have a selection probability close to $\theta$. 
\paragraph*{Choice of $q$ and threshold $\tau$.} One can regard the choice of $q = \mathbb{E}(|\hat{S}_{\floor{n/2}}|)$ (which is usually fixed through a tuning parameter $\lambda$) as part of the choice of the base selection procedure.  One option is to fix $q$ by varying $\lambda$ at each evaluation of the selection procedure until it selects $q$ variables.  However, if the number of variables selected at each iteration is unknown in advance (e.g. if $\lambda$ is fixed, or if cross-validation is used to choose $\lambda$ at each iteration), then $q$ can be estimated by $\sum_{k=1}^p \hat{\Pi}_B(k)$.


An important point to note is that although choosing $\lambda$ or $q$ is usually crucial when carrying out variable selection, this is not the case when using CPSS. Our experience is that the performance of CPSS is surprisingly insensitive to the choice of $q$ (see also \citet{MeinshausenBuhlmann2010}).  That is to say, $L_{q/p}$ does not vary much as $q$ varies, and also the final selected sets for different values of $q$ tend to be similar (where different thresholds are chosen to control the selection of variables in $L_{q/p}$ at a pre-specified level).  Thus, when using CPSS, it is the threshold $\tau$ that plays a role similar to that of a tuning parameter for the base procedure.  The great advantage of CPSS is that our bounds allow one to choose $\tau$ to control the expected number of low selection probability variables selected.

To summarise: we recommend as a sensible default CPSS procedure taking $B=50$ and $\theta = q/p$.  We then choose $\tau$ using the bound~(\ref{eq:FinalBound}) with $|L_\theta|$ replaced by $p$ to control the expected number of low selection probability variables chosen.



\section{Numerical properties}
\subsection{Simulation Study}
\label{Sec:SimStudy}
In this section we investigate the performance and validity of the bounds derived in the previous section by applying CPSS to simulated data.  We consider both linear and logistic regression and different values of $p$ and $n$. In each of these settings, we first generate independent explanatory vectors $X_1, \ldots, X_n$ with each $X_i \sim N_p(0, \Sigma)$.  We use a Toeplitz covariance matrix $\Sigma$ with entries
\[
 \Sigma_{ij} = \rho ^{||i-j| - p/2| - p/2},
\]
and we look at various values of $\rho$ in $[0,1)$.  So the correlation between the components decays exponentially with the distance between them in $\mathbb{Z}_p$.

For linear regression, we generate a vector of errors $\epsilon \sim N_n (0, \sigma^2I)$ and set
\[
 Y = X\beta + \epsilon,
\]
where the design matrix $X$ has $i^\mathrm{th}$ row $X_i ^T$.  The error variance $\sigma^2$ is chosen to achieve different values of the signal-to-noise ratio (SNR), which we define here by
\[
 \mathrm{SNR}^2 = \frac{\mathbb{E} \| X\beta \|^2}{\mathbb{E}\| \epsilon \|^2}.
\]

For logistic regression, we generate independent responses
\[
 Y_i \sim \mathrm{Bin}(1, p_i), \quad i=1,\ldots,n,
\]
where
\[
 \log \left(\frac{p_i}{1 - p_i}\right) = \gamma X_i ^T \beta.
\]
Here $\gamma$ is a scaling factor which is chosen to achieve a particular Bayes error rate.

In both cases, we fix the $p$-dimensional vector of coefficients $\beta$ to have $s \ll p$ non-zero components, $s/2$ of which we choose as equally spaced points within $[-1, -0.5]$ with the remaining $s/2$ equally spaced in $[0.5, 1]$.  The indices of the non-zero components, $S$, are chosen to follow a geometric progression up to rounding, with first term 1 and $(s+1)^{\mathrm{th}}$ term $p+1$.  The values are then randomly assigned to each index in $S$, but this choice is then fixed for each particular simulation setting.

With $\rho>0$, this setup will have several signal variables correlated amongst themselves, and also some signal correlated with noise.  In this way, the framework above includes a very wide variety of different data generating processes on which we can test the theory of the previous section.

By varying the base selection procedure, its tuning parameters, the values of $\rho$, $n$, $p$, $s$ and also the SNR and Bayes error rates, we have applied CPSS in several hundred different simulation settings.  For reasons of space, we present only a subset of these numerical experiments below, but the results from those omitted are not qualitatively different.

In the graphs which follow, we look at CPSS applied to the Lasso \citep{Tibshirani1996}, which we implemented using the package \texttt{glmnet} \citep{FHT2010} in \texttt{R} \citep{Rprog}. We follow the original stability selection procedure put forward in \citet{MeinshausenBuhlmann2010} and compare this to the method suggested by our $r$-concave bound \eqref{eq:FinalBound}.  Thus we first choose the level $l$ at which we wish to control the expected number of low selection probability variables (so we aim to have $\mathbb{E}|\hat{S}_{n,\tau}^{\mathrm{CPSS}} \cap L_{q/p}| \leq l$).  Then we fix $q = \sqrt{0.8lp}$ and set the threshold $\tau$ at 0.9. This ensures that, according to the original worst case bound, we control the expected number of low selection probability variables selected at the required level. In the $r$-concave case, we take our threshold as
\[
 \tilde{\tau} = \min\{\tau \in \{0, 1/2B, \ldots, 1\} : \min\{D(q^2/p^2, 2\tau - 1, B, -1/2), D(q/p, \tau, 2B, -1/4)\} \leq l/p\}.
\]
We also give the results one would obtain using the Lasso alone, but with the benefit of an oracle which knows the optimal value of the tuning parameter $\lambda$. That is, we take $\hat{S}^{\lambda^*} _n$ as our selected set, where
\[
 \lambda^* = \inf\{ \lambda: \mathbb{E}|\hat{S}^\lambda _n \cap L_{q/p}| \leq l\},
\]
and $\hat{S}^\lambda _n$ is the selected set when using the Lasso with tuning parameter $\lambda$ applied to the whole data set.

We present all of our results relative to the performance of CPSS using an oracle-driven threshold $\tau^*$, where $\tau^*$ is defined by
\[
 \tau^* = \min\{\tau \in \{0, 1/2B, \ldots, 1\} : \mathbb{E}|\hat{S}_{n,\tau}^{\mathrm{CPSS}} \cap L_{q/p}| \leq l \}.
\]
Referring to Figures~\ref{fig:gaussian1}-\ref{fig:binomial2}, the heights of the black bars, grey bars and crosses are given by
\begin{equation*}
 \frac{\mathbb{E}|\hat{S}_{n,0.9}^{\mathrm{CPSS}} \cap S|}{\mathbb{E}|\hat{S}_{n,\tau^*}^{\mathrm{CPSS}} \cap S|}, \,
 \frac{\mathbb{E}|\hat{S}_{n,\tilde{\tau}}^{\mathrm{CPSS}} \cap S|}{\mathbb{E}|\hat{S}_{n,\tau^*}^{\mathrm{CPSS}} \cap S|} \quad \text{and} \quad
 \frac{\mathbb{E}|\hat{S}^{\lambda^*} _n \cap S|}{\mathbb{E}|\hat{S}_{n,\tau^*}^{\mathrm{CPSS}} \cap S|},
\end{equation*}
respectively. Thus the heights of the black and grey bars relate to the loss of power in using the threshold suggested by the corresponding bounds. In all of our simulations, we used $B=50$. Each scenario was run 500 times, and in order to determine the set $L_{q/p}$, in each scenario, we applied the particular selection procedure $\hat{S}_{\floor{n/2}}$ to 50,000 independent data sets.

It is immediately obvious from the results that using the $r$-concave bound, we are able to recover significantly more variables in $S$ than when using the the worst case bound.  Furthermore, though it is not shown in the graphs explicitly, we also achieve the required level of error control in all but one case (where the $r$-concavity assumption fails). In fact the one particular example is hardly exceptional in that we have $\mathbb{E}|\hat{S}_{n,\tilde{\tau}}^{\mathrm{CPSS}} \cap L_{q/p}| = 1.034 > 1 = l$. Thus in close accordance with our theory, there are no significant violations of the $r$-concave bound.

We also see that the loss in power due to using $\tilde{\tau}$ rather than $\tau^*$, is very low. In almost all of the scenarios, we are able to select more than 75\% of the signal we could select with the benefit of an oracle, and usually much more than this.  It is interesting that the performance of the oracle CPSS and oracle Lasso procedures are fairly similar.  The key advantage of CPSS is that it allows for error control whereas there is in general no way of determining (or even approximating) the optimal $\lambda^*$ that achieves the required error control.  In fact, the performance of CPSS with our bound is only slightly worse then that of the oracle Lasso procedure, and in a few cases, particularly when $\rho$ is small, it is even slightly better. In the cases where $\rho \geq 0.75$, we see that CPSS is not quite as powerful. This is because having such large correlations between variables causes $\{p_{k, \floor{n/2}}:k = 1, \ldots, p\}$ to be relatively spread out in $[0,1]$. As explained in Section~\ref{Sec:Practice}, we expect our bound to weaken in this situation.  However, even when the correlation is as high as 0.9, we recover a sizeable proportion of the signal we would select had we used the optimal $\tau^*$.
\begin{figure}[p]
  \centering
  \includegraphics[height=\textwidth, clip = true, angle = 270, trim = 0 0 35 22]{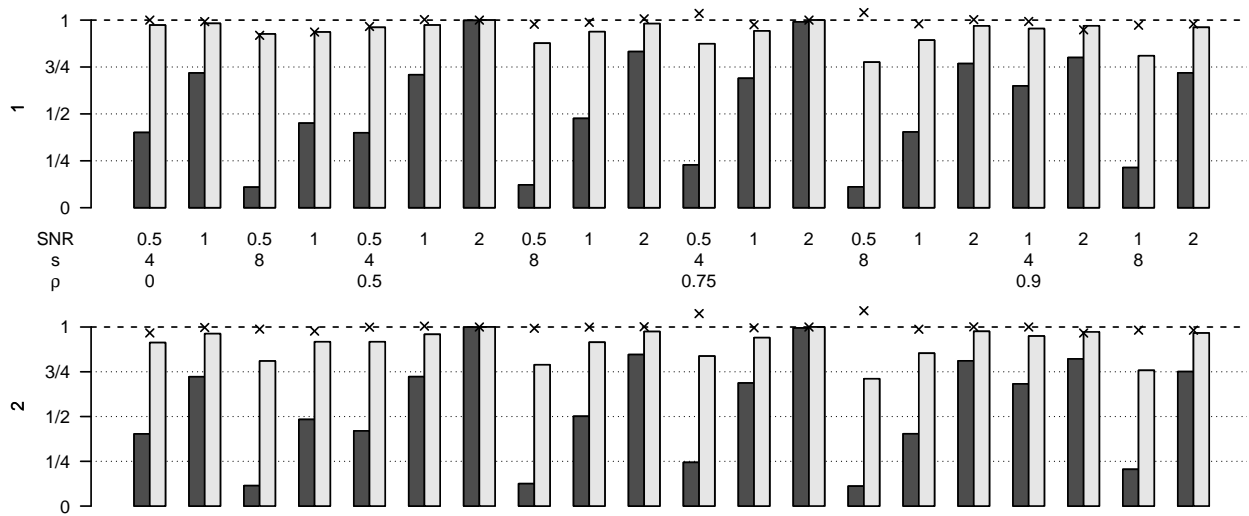}
  \caption{Linear regression with $n=200$, $p=1000$. The black and grey bars correspond to the worst case and $r$-concave procedures respectively, with higher bars being preferred. The crosses correspond to a theoretical oracle-driven Lasso procedure (see the beginning of Section~\ref{Sec:SimStudy} for further details). The $y$-axis label gives the error control level $l$.}
  \label{fig:gaussian1}
\end{figure}
\begin{figure}[pb]
  \centering
  \includegraphics[height=\textwidth, clip = true, angle = 270, trim = 0 0 35 22]{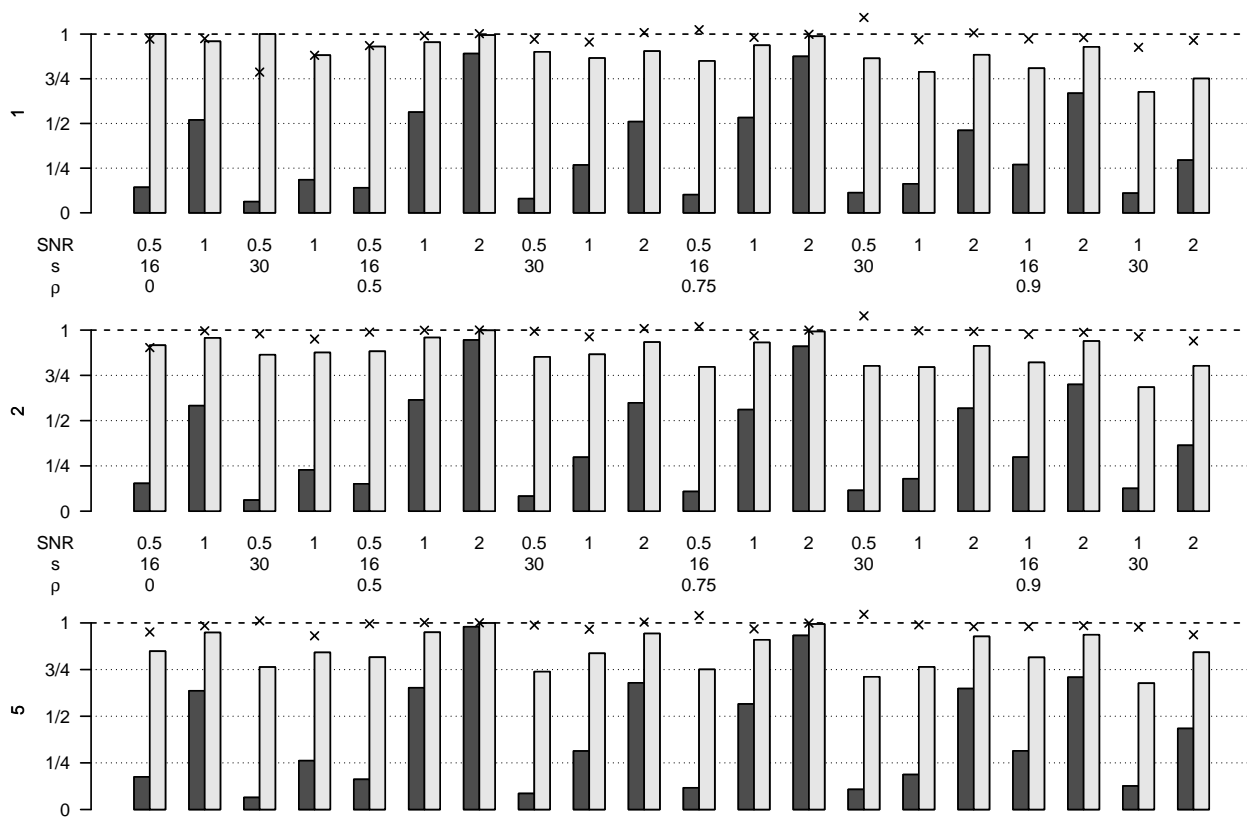}
  \caption{As above but $n=500$, $p=2000$.}
  \label{fig:gaussian2}
\end{figure}
\begin{figure}[pt]
  \centering
  \includegraphics[height=\textwidth, clip = true, angle = 270, trim = 0 0 35 22]{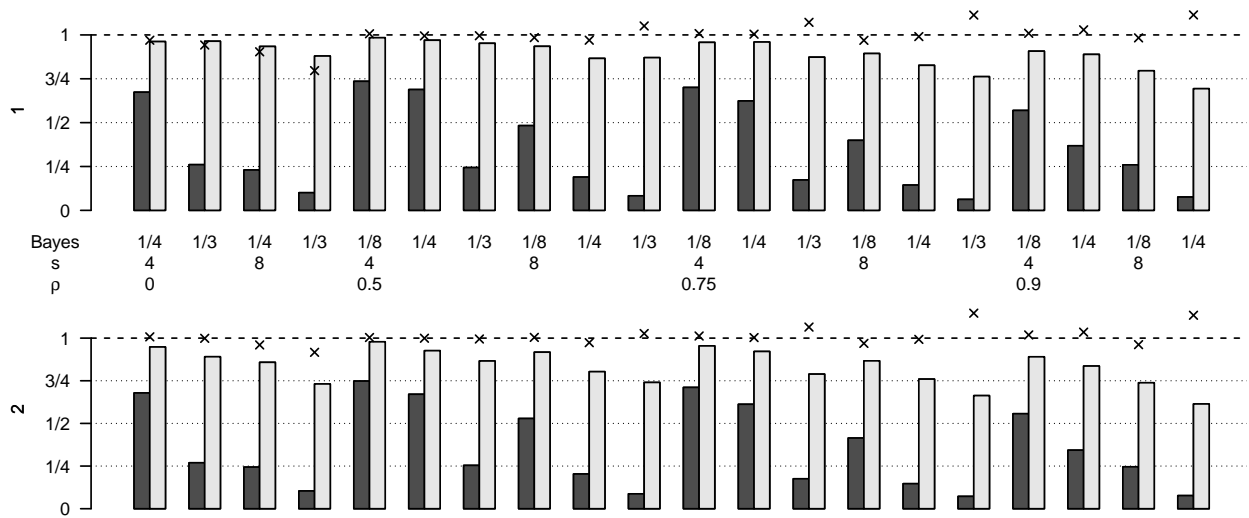}
  \caption{As Figure~\ref{fig:gaussian1} but with logistic regression (and $n=200$, $p=1000$).}
  \label{fig:binomial1}
\end{figure}
\begin{figure}[pb]
  \centering
  \includegraphics[height=\textwidth, clip = true, angle = 270, trim = 0 0 35 22]{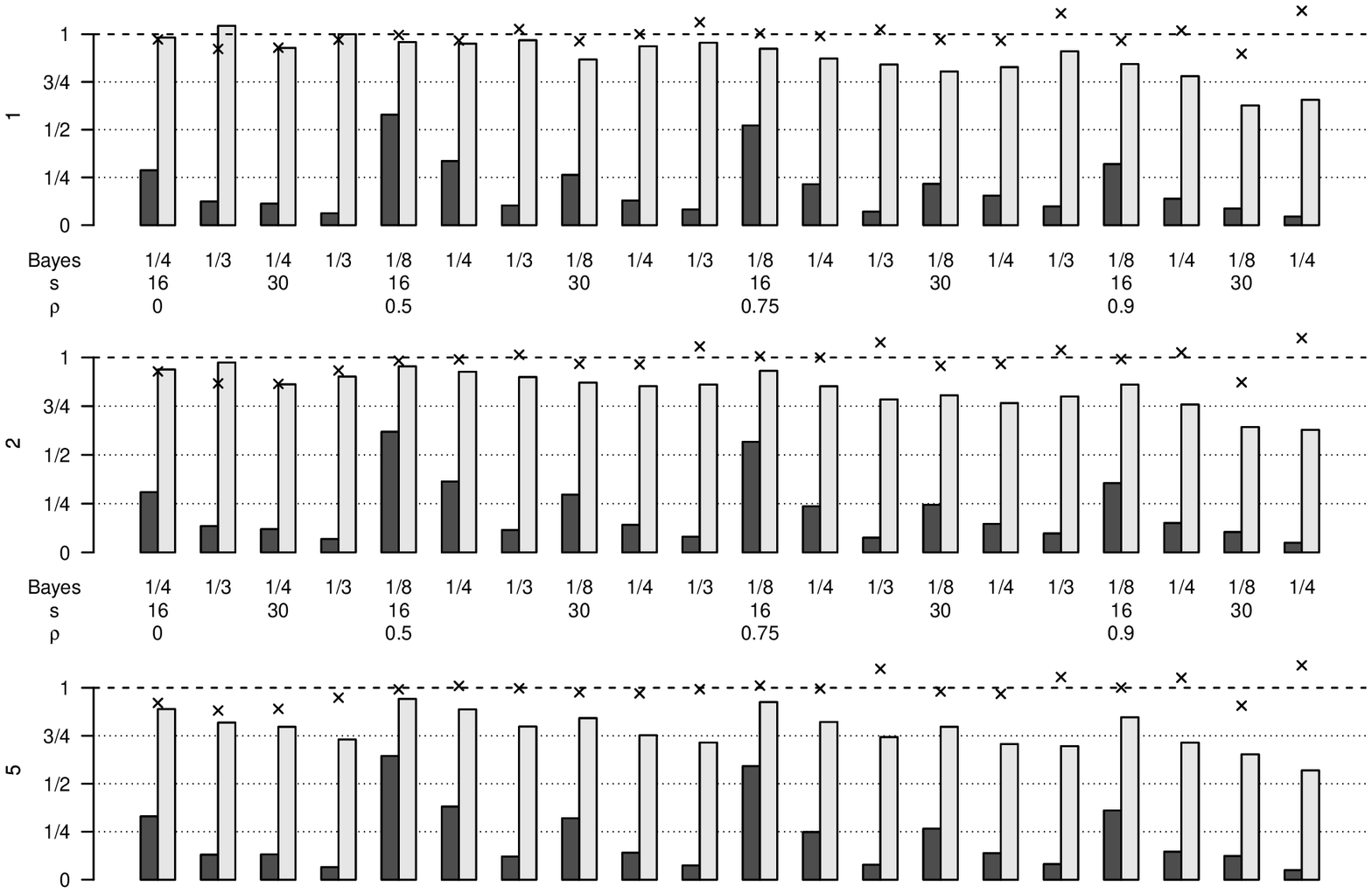}
  \caption{As above but with with $n=500$, $p=2000$.}
  \label{fig:binomial2}
\end{figure}
\subsection{Real data example}
Here we illustrate our CPSS methodology on the widely studied colon data set of \citet{Alon1999}, freely available at \texttt{http://microarray.princeton.edu/oncology/affydata/index.html}.  The data consist of 2000 gene expression levels from 40 colon tumour samples and 22 normal colon tissue samples, measured using Affymetrix oligonucleotide arrays.  Our goal is to identify a small subset of genes which we are confident are linked with the development of colon cancer.  Such a task is important for improving scientific understanding of the disease and for selecting genes as potential drug targets.

The data were first preprocessed by averaging over the expression levels for repeated genes (which had been tiled more than once on each array), log-transforming each gene expression level, standardising each row to have mean zero and unit variance, and finally removing the columns corresponding to control genes, so that $p=1908$ genes remained.  The transformation and standardisation are very common preprocessing steps to reduce skewness in the data and help eliminate the effects of systematic variations between different microarrays (see for example \citet{AmaratungaCabrera} and \citet{DFS2002}).

We applied CPSS with $\ell_1$ (Lasso) penalised logistic regression as the base procedure, with $B=50$, and choosing $\tau$ both using the $r$-concave bound of Section~\ref{Sec:Practice}, and the original bound of \citet{MeinshausenBuhlmann2010}.  We estimated the expected classification error in the two cases by averaging over 128 repetitions of stratified random subsampling validation, taking 8 cancerous and 4 normal observations in each test set. Thus when applying CPSS, we had $n = 40+22-12 = 50$. We looked at $q=8, \, 10$ and $12$, and set $\tau$ to control $\mathbb{E}|\hat{S}_{n,\tau}^{\mathrm{CPSS}} \cap L_{q/p}| \leq l$ with $l=0.1$ and $0.5$.

Rather than subsampling completely at random when using CPSS, we also stratified these subsamples to include the same proportion of cancerous to normal samples as in the training data supplied to the procedure. Without this step, some of the subsamples may not include any samples from one of the classes, and applying $\hat{S}_{\floor{n/2}}$ to such a subsample would give misleading results. Using stratified random subsampling is still compatible with our theory, provided that $\mathbb{E}|\hat{S}_{n,\tau}^{\mathrm{CPSS}} \cap L_{\theta}|$ is interpreted as an expectation over random data which contain the same class proportions as observed in the original data. In general, this approach of stratified random subsampling is useful when the response is categorical.

The results in Table~\ref{tab01} show that, as expected, the new error bounds allow one to select more variables than the conservative bounds of \citet{MeinshausenBuhlmann2010} for the same level of error control, and as a consequence, the expected prediction error is reduced.
Figure~\ref{fig:robustness} demonstrates the robustness of the selected set to the different values of $q$.  Finally, we also applied CPSS on the entire dataset with $q=8$ and $B=50$ and using the $r$-concave bound of Section~\ref{Sec:Practice} to choose $\tau$ to control $\mathbb{E}|\hat{S}_{n,\tau}^{\mathrm{CPSS}} \cap L_{q/p}| \leq 0.5$ (cf. Figure~\ref{fig:genemap}).  We see that with just 5 genes out of 1908, we manage to separate the two classes quite well.

\begin{table}
\caption{\label{tab01} Improvement in classification error (\%) over the naive classifier which always determines the data to be from a cancerous tissue. Thus the classification errors are $33\frac{1}{3}\%$ minus these quantities. We also give the average number of variables selected in parentheses.}
\centering
\fbox{%
\begin{tabular}{lcccc}
&\multicolumn{2}{c}{Worst case procedure} & \multicolumn{2}{c}{$r$-concave procedure} \\
$q$ & $l=0.1$ & $l=0.5$ & $l=0.1$ & $l=0.5$ \\
\hline
8 & 4.9 (0.5) & 11.6 (1.1) & 16 (2.3) & 17.5 (5.1) \\
10 & 0.9 (0.1) & 10.6 (0.9) & 14.7 (1.6) & 15.8 (4.4) \\
12 & 0.0 (0.0) & 9.4 (0.8) & 12.8 (1.1) & 15.8 (4.1) \\
\end{tabular}}
\end{table}

\begin{figure}
\centering
\includegraphics{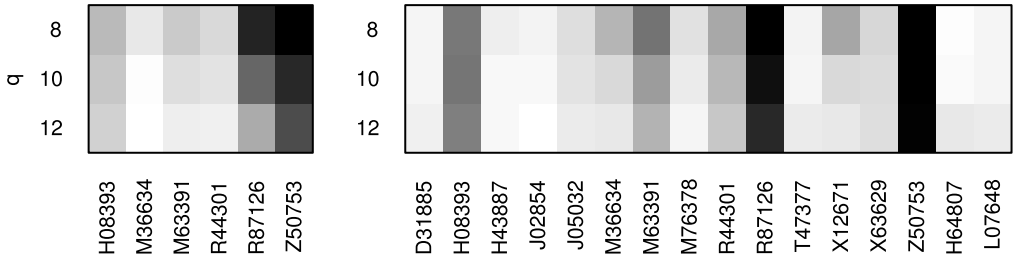}
\caption{\label{fig:robustness} For $l=0.1$ (left) and $l=0.5$ (right), we have plotted the proportion of times a gene was selected by our $r$-concave CPSS procedure for all genes which were selected at least 5\% of the time among the 128 repetitions. Solid black means the gene was selected in every repetition, and white means it was never selected. Thus dark vertical lines indicate that the choice of $q$ has little effect on the end result of CPSS.}
\end{figure}

\begin{figure}
\centering
\includegraphics[width = \textwidth]{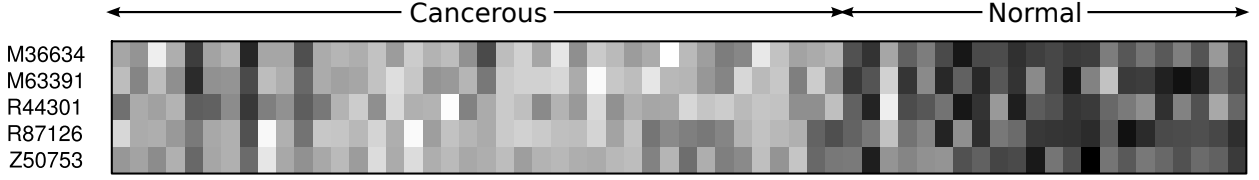}
\caption{\label{fig:genemap} A heatmap of the normalised, centered, log intensity values of the genes selected when we use the $r$-concave bound to choose $\tau$ such that we control $\mathbb{E}|\hat{S}_{n,\tau}^{\mathrm{CPSS}} \cap L_{q/p}| \leq 0.5$.}
\end{figure}

\appendix
\section{Appendix}

\subsection{Proof of Theorem~\ref{thm:hi_low}}

The proof of Theorem~\ref{thm:hi_low} requires the following lemma.
\begin{lemma}
\label{lemma:selprob}
\begin{enumerate}[(i)]
\item If $\tau \in (\frac{1}{2},1]$, then
\[
\mathbb{P}(k \in \hat{S}_{n,\tau}^{\mathrm{CPSS}}) \leq \frac{1}{2\tau-1}p_{k, \floor{n/2}}^2. 
\]
\item If $\tau \in [0,\frac{1}{2})$, then
\[
\mathbb{P}(k \notin \hat{S}_{n,\tau}^{\mathrm{CPSS}}) \leq \frac{1}{1-2\tau}(1-p_{k, \floor{n/2}})^2.
\]
\end{enumerate}
\end{lemma}
\begin{proof}
(i) Let $\mathcal{A} = \{(A_{2j-1},A_{2j}):j = 1,\ldots,B\}$ be randomly chosen independent pairs of subsets of $\{1,\ldots,n\}$ of size $\floor{n/2}$ such that $A_{2j-1} \cap A_{2j} = \emptyset$.  Then 
\begin{equation} 
\label{eq:simult}
0 \leq \frac{1}{B} \sum_{j=1}^B \bigl\{1 - \mathbbm{1}_{\{k \in \hat{S}(A_{2j-1})\}}\bigr\}\bigl\{1 - \mathbbm{1}_{\{k \in \hat{S}(A_{2j})\}}\bigr\} = 1 - 2\hat{\Pi}_B(k) + \tilde{\Pi}_B(k).
\end{equation}
Now $\mathbb{E}\{\tilde{\Pi}_B(k)\}= \mathbb{E}\{\mathbb{E}(\tilde{\Pi}_B(k) | \mathcal{A})\} = p_{k, \floor{n/2}} ^2$ because $\hat{S}(A_{2j-1})$ and $\hat{S}(A_{2j})$ are independent conditional on $\mathcal{A}$.  It follows using \eqref{eq:simult} that 
\begin{align}
\label{Eq:Markov}
\mathbb{P}(k \in \hat{S}_{n,\tau}^{\mathrm{CPSS}}) = \mathbb{P}\{\hat{\Pi}_B(k) \geq \tau\} \leq \mathbb{P}\bigl\{\tfrac{1}{2}(1+\tilde{\Pi}_B(k)) \geq \tau \bigr\} &= \mathbb{P}\{\tilde{\Pi}_B(k) \geq 2\tau -1\} \nonumber \\
 & \leq \frac{1}{2\tau - 1} p_{k, \floor{n/2}}^2,
\end{align}
where we have used Markov's inequality in the final step.  

(ii) Define $\hat{\Pi}_B^{\hat{N}_n}$ and $\tilde{\Pi}_B^{\hat{N}_n}$ by replacing $\hat{S}_n$ with $\hat{N}_n:= \{1,\ldots,p\} \setminus \hat{S}_n$ in the definitions of $\hat{\Pi}_B$ and $\tilde{\Pi}_B$ respectively.  Then, using the bound corresponding to \eqref{eq:simult} and Markov's inequality again,
\begin{align*}
\mathbb{P}(k \notin \hat{S}_{n,\tau}^{\mathrm{CPSS}}) = \mathbb{P}\{\hat{\Pi}_B(k) < \tau\} = \mathbb{P}\{\hat{\Pi}_B^{\hat{N}_n}(k) > 1 - \tau\} &\leq \mathbb{P}\{\tilde{\Pi}_B^{\hat{N}_n}(k) > 1-2\tau\} \\
&\leq \frac{1}{1 - 2\tau} (1-p_{k, \floor{n/2}})^2 . 
\end{align*}
\end{proof}

\begin{prooftitle}{of Theorem~\ref{thm:hi_low}}
(i) Note that
\[
\mathbb{E}|\hat{S}_{\floor{n/2}} \cap L_\theta| = \mathbb{E}\biggl(\sum_{k=1}^p \mathbbm{1}_{\{k \in \hat{S}_{\floor{n/2}}\}} \mathbbm{1}_{\{p_{k,\floor{n/2}} \leq \theta\}}\biggr) = \sum_{k=1}^p p_{k, \floor{n/2}}\mathbbm{1}_{\{p_{k,\floor{n/2}} \leq \theta\}}.
\]
By Lemma~\ref{lemma:selprob}, it follows that 
\begin{align*}
\mathbb{E}|\hat{S}_{n,\tau}^{\mathrm{CPSS}} \cap L_\theta| = \mathbb{E}\biggl( \sum_{k=1}^p \mathbbm{1}_{\{k \in \hat{S}_{n,\tau}^{\mathrm{CPSS}}\}} & \mathbbm{1}_{\{p_{k,\floor{n/2}} \leq \theta\}}\biggr) = \sum_{k=1}^p \mathbb{P}(k \in \hat{S}_{n,\tau}^{\mathrm{CPSS}}) \mathbbm{1}_{\{p_{k,\floor{n/2}} \leq \theta\}} \\
&\hfill \leq \frac{1}{2\tau-1}\sum_{k=1}^p p_{k, \floor{n/2}}^2 \mathbbm{1}_{\{p_{k,\floor{n/2}} \leq \theta\}} \leq \frac{\theta}{2\tau -1}\mathbb{E}|\hat{S}_{\floor{n/2}} \cap L_\theta|. 
\end{align*}
(ii) This proof is very similar to that of (i) and is omitted. $\hfill \Box$
\end{prooftitle}

\subsection{Proof of Theorem~\ref{thm:Unimodal}}

The proof of Theorem~\ref{thm:Unimodal} requires several preliminary results, and we use the following notation.  Let $G$ denote the finite lattice $\{0,\tfrac{1}{B},\tfrac{2}{B},\ldots,1\}=\tfrac{1}{B} \mathbb{Z} \cap [0,1]$.  If $f$ is a probability mass function on $G$, we write $f_i$ for $f(i/B)$, thereby associating $f$ with $(f_0,f_1,\ldots,f_B) \in \mathbb{R}^{B+1}$.

For $t \in G$, we denote the probability that a random variable distributed according to $f$ takes values greater than or equal to $t$ by $\mathcal{T}_t(f):=\sum_{i \geq Bt} f_i$.  We also write $\mathcal{E}(f):=\sum_{i=1}^B \frac{i}{B}f_i$ for the expectation of this random variable and $\mathrm{supp}(f) := \{i/B \in G:f_i > 0\}$ for the support of $f$.


%
Let $\mathcal{U}$ be the set of all unimodal probability mass functions $f$ on $G$, and let $\mathcal{U}_\eta = \{f \in U:\mathcal{E}(f) \leq \eta\}$.   We consider the problem of maximising $\mathcal{T}_t$ over $f \in \mathcal{U}_\eta$.  Since the cases $\eta = 0$ and $t \leq \eta$ are trivial, there is no loss of generality in assuming throughout that $0 < \eta < t$ and $t \in G$, so in particular $t \geq 1/B$.   
\begin{lemma} 
\label{lemma:u_exists}
There exists a maximiser of $\mathcal{T}_t$ in $\mathcal{U}_\eta$.
\end{lemma}
\begin{proof}
Since $\mathcal{T}_t:\mathbb{R}^{B+1} \to \mathbb{R}$ is linear and therefore continuous, it suffices to show that $\mathcal{U}_\eta \subset \mathbb{R}^{B+1}$ is closed and bounded.  Now $\mathcal{U}_\eta$ is bounded as $\mathcal{U}_\eta \subset [0,1]^{B+1}$.  Moreover, the hyperplane $H = \{(x_0,\ldots,x_B): x_0 + x_1 + \ldots + x_B = 1\}$ is closed.  Also, $\mathcal{E}$ is a continuous function on $\mathbb{R}^{B+1}$, so $\mathcal{E}^{-1}([0,\eta])$ is closed.  Now let $O=\{f \in \mathbb{R}^{B+1} : f \text{ is not unimodal}\}$.  If $f \in O$ then there must exist $i_1 < i_2 < i_3$ such that $f_{i_2} < \min\{f_{i_1}, f_{i_3}\}$.  Clearly this inequality must hold for all $g$ in a sufficiently small open ball about $f$, so $O$ is open. We see that
\[
\mathcal{U}_\eta = H \cap \mathcal{E}^{-1}([0,\eta]) \cap O^c .
\]
Thus $\mathcal{U}_\eta$ is an intersection of closed sets and hence is closed.
\end{proof}
We will make frequent use of the following simple proposition in subsequent proofs.
\begin{prop} 
\label{prop:ineq}
Suppose that $(x_1,\ldots,x_n) \in \mathbb{R}^n$ and $(y_1,\ldots,y_n) \in \mathbb{R}^n$ satisfy
\[
\sum_{i=1} ^n x_i = \sum_{i=1} ^n y_i,
\]
and that there exists some $i^* \in \{1,\ldots, n\}$ with $x_i \geq y_i$ for all $i \leq i^*$ and $x_i \leq y_i$ for all $i > i^*$. Then
\[
 \sum_{i=1} ^n i x_i \leq \sum_{i=1} ^n i y_i,
\]
with equality if and only if $x_i=y_i$ for $i=1,\ldots,n$.
\end{prop}
\begin{proof}
We have
\[
\sum_{i \leq i^*} i(x_i - y_i) \leq i^* \sum_{i \leq i^*} (x_i - y_i) = i^* \sum_{i > i^*} (y_i - x_i) \leq \sum_{i > i^*} i(y_i - x_i). 
\]
\end{proof}
The following result characterises the extremal elements of $\mathcal{U}_\eta$ in the sense of maximising the tail probability $\mathcal{T}_t$.  In particular, it shows that such extremal elements can take only one of two simple forms.
\begin{prop} \label{prop:ubound}
Any maximiser $f^* \in \mathcal{U}_\eta$ of $\mathcal{T}_t$ satisfies
\begin{enumerate}[(i)]
  \item $\mathcal{E}(f^*) = \eta$,
  \item writing $i_M$ for $B\max(\supp(f^*))$, we have either 
  \begin{enumerate}[(a)]
    \item $f^* _0 > f_1^* = f_2^* = \ldots = f_{i_{M-1}}^* \geq f^* _{i_M}$, or
    \item $i_M = t$ and $f^* _0 = f_1^* = \ldots = f_{i_{M-1}}^* \leq f^* _{i_M}$.
  \end{enumerate}
\end{enumerate}
\end{prop}
\begin{proof}
(i) Suppose $f^* \in \mathcal{U}_\eta$ maximises $\mathcal{T}_t$, but that $\mathcal{E}(f^*)<\eta$. Define $i_m:=\min(\supp(f^*))$. As $\eta<\tau$, we must have $i_m<Bt$. Define $g$ by
\[
g_i= \left\{ \begin{array}{ll}
       0 & \text{ if } i<i_m \\
      f^*_i - \epsilon_1 & \text{ if } i=i_m \\
    f^*_i + \epsilon_2 & \text{ if } i>i_m \\
      \end{array} \right.
\]
where $\epsilon_1, \, \epsilon_2 >0$ are chosen such that $\sum_{i=0} ^B g_i = 1$, but are small enough that $\mathcal{E}(g) \leq \eta$. Then $g \in \mathcal{U}_\eta$ but $\mathcal{T}_t(g)>\mathcal{T}_t(f^*)$, a contradiction.

(ii) Suppose first that there exists a mode of $f^*$ which is at least $t$.  Let $g \in \mathcal{U}_\eta$ be such that $g_i=f^*_i$ for $i \geq Bt$ and $g_i= \frac{1}{Bt}\sum_{\ell=0}^{Bt-1} f_\ell^*$ for $i < Bt$. As $f_0^* \leq f_1^* \leq \ldots \leq f_{Bt}^*$, we can apply Proposition~\ref{prop:ineq} to see that
\begin{equation} \label{eq:exp1}
 \mathcal{E}(g) \leq \mathcal{E}(f^* ).
\end{equation}
But $\mathcal{T}_t(g) = \mathcal{T}_t(f^*)$, so by optimality of $f^*$ we must have equality in \eqref{eq:exp1}. Thus Proposition~\ref{prop:ineq} gives us that $f^* = g$. 

Next, define $h \in \mathcal{U}_\eta$ by $h_i=f^* _i$ for $i < Bt$, $h_{Bt}= \mathcal{T}_t(f^*)$, and $h_i = 0$ for $i > Bt$. Then $\mathcal{T}_t(h) = \mathcal{T}_t(f^*)$.  Again Proposition~\ref{prop:ineq} and the optimality of $f^*$ give that $f^* = h$. Thus $f^*$ satisfies property (ii)(b) of the theorem. 

Now suppose that there is no mode of $f^*$ which is at least $t$, so $f_{Bt}^* \geq f_{Bt+1}^* \geq \ldots \geq f_B^*$. Let $g \in \mathcal{U}_\eta$ satisfy $g_i=f^*_i$ for $i \geq Bt$ and $g_1 = \ldots = g_{Bt}$.  We must have $g_0 > g_1$, otherwise $f^*$ would have a mode at $t$. As $\mathcal{T}_t(g) = \mathcal{T}_t(f^*)$, optimality of $f^*$ and Proposition~\ref{prop:ineq} imply $f^*= g$.

Finally, let $h \in \mathcal{U}_\eta$ satisfy $h_i=f^* _i$ for $i \leq Bt$ and $h_{Bt} = h_{Bt+1} = \ldots = h_{k-1} \geq h_k$, where $k$ and $h_k$ are chosen such that $\sum_{i=0} ^B h_i =1$.  As before, Proposition~\ref{prop:ineq} allows us to deduce that $f^* = h$.  Thus $f^*$ satisfies property (ii)(a) of the theorem.
\end{proof}
We are now in a position to state Markov's inequality for random variables with unimodal distributions on $G$, which may be of some independent interest.
\begin{thm}[Markov's inequality under unimodality]
\label{thm:UniMarkov}
Let $X$ be a random variable with a unimodal distribution on $G = \{0,\tfrac{1}{B},\tfrac{2}{B},\ldots,1\}$, and let $t \in G$.  If $\eta := \mathbb{E}(X) \leq 1/3$, then
\[
\mathbb{P}(X \geq t) \leq \left\{ \begin{array}{ll} \dfrac{2\eta - t + \frac{1}{B}}{t + \frac{1}{B}} & \text{ if } \;\; t \in \left(\eta, \, \min\left(\tfrac{3}{2}\eta + \tfrac{1}{2B},2\eta\right) \right] \\
			
				    \dfrac{\eta}{2t-\frac{1}{B}} & \text{ if } \;\; t \in \left(\min\left(\tfrac{3}{2}\eta + \tfrac{1}{2B}, 2\eta\right) , \, \tfrac{1}{2}\right] \\
                        
                              \dfrac{2\eta(1- t + \frac{1}{B})}{1 + \frac{1}{B}} & \text{ if } \;\; t \in \left(\tfrac{1}{2}, \, 1\right].
                             \end{array} \right.
\]
Let $d$ be defined by
\[
d:= d(\eta, B) = -2\left(\eta - \tfrac{1}{2}\right)(6\eta + 1) + \frac{2-4\eta}{B} + \frac{(4\eta - 1)^2}{B^2}.
\]
If $\eta > 1/3$ and $d > 0$, then
\begin{equation*}
\mathbb{P}(X \geq t) \leq \begin{cases} \dfrac{2\eta - t + \frac{1}{B}}{t + \frac{1}{B}} & \text{ if } \;\; t \in \left( \eta, \, \tfrac{1}{2} + \tfrac{1}{4\eta} (1 + \tfrac{1}{B} - d^{1/2}) \right] \\
                                                              \dfrac{2\eta(1- t + \frac{1}{B})}{1 + \frac{1}{B}} & \text{ if } \;\; t \in \left(\tfrac{1}{2} + \tfrac{1}{4\eta}( 1 + \tfrac{1}{B} - d^{1/2}), \, 1\right].
                             \end{cases}
\end{equation*}
Finally, if $\eta > 1/3$ and $d \leq 0$, then
\[
\mathbb{P}(X \geq t) \leq \dfrac{2\eta - t + \frac{1}{B}}{t + \frac{1}{B}}.
\]
\end{thm}
\begin{proof}
Proposition~\ref{prop:ubound} tells us that $\mathbb{P}(X \geq t)$ must be at most the maximum of the optimal solutions to the following two optimisation problems:
\[
\begin{tabular}{llllll}
$(P)$: & Maximise & $b(s-Bt) + c$ in $a,b,c,s$ \quad &$(Q)$: &Maximise & $b$ in $a, b$ \\
& subject to & $a + (s-1)b + c = 1$  & & subject to & $Bta + b = 1$ \\
& & $\frac{s}{2}(s-1)b + sc = B\eta$ & & & $\frac{Bt}{2}(Bt-1)a + Btb = B\eta$ \\
& & $a > b \geq c \geq 0$  & & & $b \geq a \geq  0$. \\
& & $s \in \{Bt, Bt+1, \ldots, B\}$ & & & 
\end{tabular}
\]
Problem $(P)$ corresponds to case (ii)(a) of Proposition~\ref{prop:ubound}, and problem $(Q)$ to case (ii)(b).

The solution to $(Q)$ is determined entirely by the constraints, and we see that the optimal value is
\begin{equation}
 \frac{2\eta - t + \frac{1}{B}}{t + \frac{1}{B}}.
\end{equation}
 
To solve $(P)$, we break it into $B(1-t)+1$ subproblems: for $s \in \{Bt, Bt+1, \ldots, B\}$, we define subproblem $(P(s))$ as follows:
\[
\begin{tabular}{lll}
$(P(s))$: & Maximise & $b(s-Bt) + c$ in $a,b,c$ \\
&subject to & $a + (s-1)b + c = 1$ \\
&& $\frac{s}{2}(s-1)b + sc = B\eta$ \\
&& $b \geq c$, \\
&& $a, b, c \geq 0.$
\end{tabular}
\]
Notice that we have not included the $a > b$ constraint.  This is because Proposition~\ref{prop:ubound} ensures that this constraint is always satisfied at an optimal solution of $(P)$, so there exists $s^*$ such that every optimal solution of $(P(s^*))$ corresponds to an optimal solution of $(P)$. 

Now each subproblem is a standard linear programming problem, so we know that one of the basic feasible solutions must be optimal.  Since $a > 0$, all basic feasible solutions must have either $c=0$ or $b=c$. Thus we may replace the subproblems $(P(s))$ by
\[
\begin{tabular}{lll}
$(P'(s))$: & Maximise & $b(s-Bt +1)$ in $a,b$ \\
&subject to & $a + sb = 1$ \\
&& $\frac{s}{2}(s + 1)b = B\eta$ \\
&& $a, b \geq 0$.
\end{tabular}
\]
The second constraint is enough to determine that the optimal value of $P'(s)$ is
\begin{equation} \label{eq:objfunc}
\frac{2B\eta(s-Bt+1)}{s(s+1)}=: \gamma(s).
\end{equation}
Now we can proceed to find an $s^*$ which maximises $\gamma$ over $\{Bt,Bt+1,\ldots,B\}$. The sign of $\gamma'(s)$ is the sign of
\[
-s^2+2(Bt -1)s + Bt - 1.
\]
This quadratic in $s$ has roots
\[
Bt -1 \pm\sqrt{(Bt - 1)^2+ Bt -1}.
\]
So $\gamma(s)$ is increasing for all $s \in \{Bt,Bt+1,\ldots,B\}$ with
\begin{equation} \label{eq:alpha}
s \leq Bt -1 +\sqrt{\left(Bt - \tfrac{1}{2}\right)^2-\tfrac{1}{4}}=: s_0.
\end{equation}
When $s_0 < B$, we must have $s^* \in \{2Bt - 2,2Bt - 1\}$. In fact, by examining \eqref{eq:objfunc}, we see that $\gamma(2Bt-2) = \gamma(2Bt-1)$. Also, from \eqref{eq:alpha}, we see that when $t > 1/2$, we have that $s_0 \geq B$, so $s^*=B$. So far, we have shown that
\[
\mathbb{P}(X \geq t) \leq \max(b_1, b_2, b_3),
\]
where bounds $b_1, b_2$ and $b_3$ are given by
\begin{align*}
b_1 := b_1 (t, \eta, B) &= \frac{2\eta - t + \frac{1}{B}}{t + \frac{1}{B}} \mathbbm{1}_{\{ \eta < t \leq \min (2\eta,1) \}} \\
b_2 := b_2 (t, \eta, B) &= \frac{\eta}{2t-\frac{1}{B}} \, \mathbbm{1}_{\{\eta < t \leq 1/2\}} \\
b_3 := b_3 (t, \eta, B) &= \dfrac{2\eta(1- t + \frac{1}{B})}{1 + \frac{1}{B}}\, \mathbbm{1}_{\{\max(\eta,1/2) \leq t \leq 1\}}.
\end{align*}

All that remains now is to determine which of $b_1, b_2$ and $b_3$ have the largest value.  We first consider the case when $\eta \leq \frac{1}{3}$.  When $t \leq \min(1/2,2\eta)$, 
\[
 \mathrm{sgn}(b_2 - b_1) = \mathrm{sgn}\left\{ \left( t - \tfrac{3}{2} \eta - \tfrac{1}{2B} \right) \left( t - \tfrac{1}{B}\right) \right\}.
\]
Now for $1/2 < t \leq 2\eta$,
\[
\frac{\partial b_3}{\partial t} = -\frac{2\eta}{1 + \frac{1}{B}} \geq -\frac{(2\eta+\frac{2}{B})}{(t + \frac{1}{B})^2} = \frac{\partial b_1}{\partial t}.
\]
Furthermore,
\[
 b_3 \left( \tfrac{1}{2} + \tfrac{1}{2B}, \eta, B \right) = \eta \geq \frac{2\eta - \frac{1}{2} + \frac{1}{2B}}{\frac{1}{2} + \frac{3}{2B}} = b_1 \left( \tfrac{1}{2} + \tfrac{1}{2B}, \eta, B \right).
\]
Putting this together gives the required bound for $\eta \leq 1/3$.

When $\eta > 1/3$, we can ignore $b_2$ as it is dominated by $b_1$. Comparing $b_1$ and $b_3$, we get the final cases of the bound.
\end{proof}
\begin{prooftitle}{of Theorem~\ref{thm:Unimodal}}
Recalling that $\mathbb{E}\{\tilde{\Pi}_B(k)\} = p_{k,\floor{n/2}}^2$, we follow the proof of Lemma~\ref{lemma:selprob}, but apply Theorem~\ref{thm:UniMarkov} at the last step of \eqref{Eq:Markov} with $t = 2\tau - 1$ to deduce that if the distribution of $\Tilde{\Pi}_B(k)$ is unimodal, then
\[
\mathbb{P}(k \in \hat{S}_{n,\tau}^{\mathrm{CPSS}}) \leq \mathbb{P}\{\tilde{\Pi}_B(k) \geq 2\tau -1\} \leq C(\tau,B)p_{k,\floor{n/2}}^2, 
\]
where $C(\tau,B)$ is given in the statement of Theorem~\ref{thm:Unimodal}.  The bound for $\mathbb{E}|\hat{S}_{n,\tau}^{\mathrm{CPSS}} \cap L_\theta|$ then follows in the same way that Theorem~\ref{thm:hi_low} follows from Lemma~\ref{lemma:selprob}. \hfill $\Box$
\end{prooftitle}

\subsection{Proofs of results on $r$-concavity}

\begin{prooftitle}{of Proposition~\ref{prop:log_iff._r}}
Suppose that $f$ is log-concave, so we may write $f=e^{-\phi}$ where $\phi$ is a convex function. If $r < 0$, then $-r\phi$ is convex, and as the exponential function is increasing and convex, $f^r = e^{-r\phi}$ is convex.  

Conversely, suppose that $f$ is not log-concave, so there exist $x, y$ and $\lambda \in (0,1)$ with
$f(\lambda x + (1-\lambda)y) < f(x)^\lambda f(y)^{1 - \lambda}$.
 Then as $M_r(f(x), f(y); \lambda) \to f(x)^\lambda f(y)^{1 - \lambda}$ as $r \to 0$, we must have $f(\lambda x + (1-\lambda)y) <M_r(f(x), f(y); \lambda)$ for some $r<0$, and so $f$ cannot be $r$-concave. \hfill $\Box$
\end{prooftitle}

\begin{prooftitle}{of Proposition~\ref{prop:uni_implies_r}}
Let $I=\{1, \ldots, l\} \cup \{u, \ldots, B-1\}$. The conditions on $f$ imply that
\[
 f_i > \min \{f_{i-1}, f_{i+1}\}, \quad i \in I.
\]
Then as $M_r (f_{i-1}, f_{i+1}, \tfrac{1}{2} ) \to \min \{f_{i-1}, f_{i+1}\}$ as $r \to -\infty$, for each $i \in I$, may choose an $r_i < 0$ with
\begin{equation} \label{eq:strict_r}
f_i > M_{r_i} (f_{i-1}, f_{i+1} ; \tfrac{1}{2} ).
\end{equation}
Set $r=\min_{i \in I} r_i$. Observe that as $M_r(a, b; \tfrac{1}{2})$ is increasing in $r$ for all fixed $a$ and $b$, the inequalities \eqref{eq:strict_r} are all satisfied when $r_i = r$. Thus $f_i ^r \leq \tfrac{1}{2}(f_{i-1} ^r + f_{i+1} ^r)$ for all $i \in \{1, \ldots, B-1\}$, so $f$ is $r$-concave. \hfill $\Box$
\end{prooftitle}
By analogy with the unimodal case, let $\mathcal{F}_{r,\eta} = \{f \in \mathcal{F}_r:\mathcal{E}(f) \leq \eta\}$.  In maximising $\mathcal{T}_t$ over $\mathcal{F}_{r,\eta}$, there is again no loss of generality in assuming $0 < \eta < t$.
\begin{lemma} \label{lem:r_exists}
For each $r < 0$, there exists a maximiser of $\mathcal{T}_t$ in $\mathcal{F}_{r,\eta}$.
\end{lemma}
\begin{proof}
This proof is almost identical to that of Lemma~\ref{lemma:u_exists}, except here we let $O=\{f \in \mathbb{R}^{B+1} : f^r \text{ is not convex}\}$.  If $f \in O$, then there must exist $i_1 < i_2 < i_3$ such that
\[
(i_3-i_2)f_{i_1}^r + (i_2-i_1)f_{i_3}^r < (i_3-i_1)f_{i_2}^r
\]
and it is clear that the above inequality must hold for all $g$ in a sufficiently small open ball about $f$. Thus $O$ is open, and the rest of the proof is clear.
\end{proof}
\begin{prop} 
\label{prop:rbound}
Any maximiser $f^* \in \mathcal{F}_{r,\eta}$ of $\mathcal{T}_t$ satisfies
\begin{enumerate}[(i)]
 \item $\mathcal{E}(f^*) = \eta$
\item $f^{*r}$ is linear between $f_0^{*r}$ and $f_{i_{M-1}}^{*r}$, where $i_M=B\max(\supp(f^*))$.   
\end{enumerate}
\end{prop}
\begin{proof}
(i) Suppose that $\mathcal{E}(f^*)<\eta$. Define $i_m:=B\min(\supp(f^*))$. Let $\phi={f^*}^r$ and define a new sequence $\psi:=(\psi_i:i=0,\ldots,B)$ by
\[
 \psi_i= \left\{ \begin{array}{ll}
       \infty & \text{ if } i<i_m \\
      \phi_i + \epsilon_1 & \text{ if } i=i_m \\
    \phi_i - \epsilon_2 & \text{ if } i>i_m \\
      \end{array} \right.
\]
where $\epsilon_1, \, \epsilon_2 >0$ are chosen such that $\sum_{i=0} ^B \psi_i^{1/r} = 1$, but are small enough that $\mathcal{E}(\psi^{1/r}) \leq \eta$. Then $\psi$ is convex, so $\psi^{1/r} \in \mathcal{F}_{r,\eta}$. Since $\eta > 0$, we must have $\mathcal{T}_t(f^*)>0$ so $\max(\supp(f^*)) \geq t$. Also, as we are assuming $\eta < \tau$, we must have $i_m<t$. Therefore $\mathcal{T}_t(\psi^{1/r})>\mathcal{T}_t(f^*)$, which is a contradiction. 

(ii) Set $\phi=f^{*r}$, so $\phi$ is convex and $\phi^{1/r}=f^*$. Define $\psi'=(\psi'_0,\ldots,\psi_B') \in \mathbb{R}^{B+1}$ as follows.  Take $\psi'_i = \phi_i$ for $i \geq Bt$, but make $\psi'$ linear between $\psi_0'$ and $\psi_{Bt}'$ such that $g:={\psi'}^{1/r}$ has $\sum_{i=0}^B g_i= 1$ and $g_0 > 0$. This is possible since $\mathcal{E}(f^*)\leq \eta < t$, so $\min(\supp(f^*))<t$.  Note that $\psi'$ is still convex since we must have $\psi'_{Bt} - \psi'_{Bt-1} \leq \phi_{Bt} - \phi_{Bt-1}$. Also $\mathcal{T}_t(g)=\mathcal{T}_t(f^*)$.  Applying Proposition~\ref{prop:ineq}, we see that $\mathcal{E}(g) \leq \mathcal{E}(f^*)$.  Optimality of $f^*$ means that equality must hold, so $f^* = g$ and also $\phi = \psi'$.

Now if $\phi$ is in fact linear between $\phi_0$ and $\phi_B$, condition (ii) of the theorem is satisfied and we are done.  Otherwise we may assume $\phi$ is not a linear function between $\phi_{Bt-1}$ and $\phi_B$ and we can define $\psi$ such that $\psi_i=\phi_i$ for $i \leq Bt$, that $\psi$ is linear between $\psi_{Bt-1}$ and $\psi_{k-1}$ and $\psi_i = \infty$ for $i > k$.  Here, $k$ is chosen such that $g:=\psi^{1/r}$ has $\sum_{i=0} ^B g_i = 1$, and the convexity of $\phi$ ensures that such a $k \leq B$ exists.  Applying Proposition~\ref{prop:ineq}, we see that $\mathcal{E}(g) \leq \mathcal{E}(f^*)$.  Since $\mathcal{T}_t(g) = \mathcal{T}_t(f^*)$, as before, optimality of $f^*$ allows us to conclude that $f^* = g$.
\end{proof}

\subsection{Computing the $r$-concave tail probability bound}
\label{Sec:Alg}
 
Here we describe a numerical algorithm that computes the function $D$ defined in Section~\ref{sec:r_concave}. Note that this is the maximum of $\mathcal{T}_t(f)$ over $f \in \mathcal{F}_{r,\eta}$.  We shall only discuss the case where $f^*$ is decreasing, as is always the case when $t > 2\eta$. The increasing case is very similar and less important for our application.  We first note that we may parametrise the $r$-concave probability mass functions whose $r^\mathrm{th}$ powers are linear as follows:
\begin{equation} \label{eq:linear_f}
f_{a,k; i} = \frac{(a + i)^{1/r}}{\sum_{j=0} ^k (a + j)^{1/r}}, \quad i=0,1, \ldots, k
\end{equation}
where $k \leq B$.  As $\mathcal{E}(f_{a,k})$ is strictly increasing in $a$, for each $k$, there is a unique $a_k$ for which $\mathcal{E}(f_{a_k,k}) = \eta$. We also note here that $a_k$ decreases with $k$.  This is easily seen by observing that, regardless of the value of $k$, the parameter $a$ in \eqref{eq:linear_f} determines the ratio of $f_{a,k;i}$ to $f_{a,k;j}$, each $i, j$.

According to Proposition~\ref{prop:rbound}, if $f^* \in \mathcal{F}_{r,\eta}$ maximises $\mathcal{T}_t$, then $f^{*r}$ is linear up to its penultimate support point. We can parametrise these in the following way.  Write
\[
\frac{\sum_{i=1} ^k i(a + i)^{1/r} + (k+1)c}{\sum_{j=0} ^k (a + j)^{1/r} + c} = B\eta,
\]
and then solve for $c$:
\[
c = c(a,k) = \frac{B\eta \sum_{j=0} ^k (a + j)^{1/r} - \sum_{i=1} ^k i(a + i)^{1/r}}{k + 1 - B\eta}.
\]
We see that as $a$ ranges through $[a_{k+1}, a_k]$, we obtain all the relevant probability mass functions supported on $0,1,\ldots, k + 1$ via
\begin{align*}
g_{a,k;i} &= \frac{(a + i)^{1/r}}{\sum_{j=0} ^k (a+j)^{1/r} + c(a, k)}, \quad i = 0,1,\ldots, k \\
g_{a,k;k+1} &= \frac{c(a, k)}{\sum_{j=0} ^k (a+j)^{1/r} + c(a, k)}.
\end{align*}
The tail probability of $g_{a,k}$, when the threshold is $t$, is
\begin{equation} \label{eq:opt}
 \mathcal{T}_t (g_{a,k}) = 1 - \frac{(k + 1 - B\eta) \sum_{i=0} ^{Bt-1} (a + i)^{1/r}}{\sum_{i=0} ^k (k + 1 - i)(a + i)^{1/r}}
\end{equation}
and we may maximise this over $a \in [a_{k+1}, a_k]$ to obtain an optimal $a^* _k$ for each $k$. This is easily accomplished using a general purpose optimiser such as \texttt{optimize} in \texttt{R}. To summarise, we have the following simple procedure for computing $\mathcal{T}_t (f^*)$.
\begin{enumerate}
 \item For each $k \in \{t, \ldots, B\}$, determine (numerically), the solution in $a_k$ to $\mathcal{E}(f_{a,k}) = \eta$.
 \item Find $a^* _k := \argmax_{a \in [a_{k+1}, a_k]} \mathcal{T}_t (g_{a,k})$, for each $k$.
 \item Let $k^* (t) := \argmax_k \mathcal{T}_t (g_{a^* _k ,k})$.
\end{enumerate}
Then $\mathcal{T}_t (f^*) = \mathcal{T}_t (g_{a^* _{k^*(t)} , k^*(t)})$.  When we wish to evaluate $\mathcal{T}_t (f^*)$ for a range of values of $t$, the process is simplified by the observation that $k^*(t)$ is increasing in $t$, and thus in Step 2 we need only consider those $k$ which are at least $k^* (t - 1/B)$.


Using the algorithm described above, we have computed
\[
 \min\{D(\theta^2, 2\tau - 1, 50, -1/2), D(\theta, \tau, 100, -1/4)\}
\]
over a grid of $\theta$ and $\tau$ values (cf. Tables~\ref{tab:r_conc1} and~\ref{tab:r_conc2}).  An \texttt{R} implementation of the algorithm is available from both authors' websites. 

\begin{table}
\caption{\label{tab:r_conc1}Table of values of $\min\{D(\theta^2, 2\tau - 1, 50, -1/2), D(\theta, \tau, 100, -1/4)\}$ for $\theta \in \{0.01,0.02,0.03,0.04,0.05\}$.}
\centering
{\scriptsize
\fbox{%
\begin{tabular}{cccccc}
& \multicolumn{5}{c}{$\theta$} \\
\cline{2-6}
$\tau$ & $0.01$ & $0.02$ & $0.03$ & $0.04$ & $0.05$ \\ 
  \hline
$0.30$ & $6.11\times 10^{-4}$ & $2.70\times 10^{-3}$ & $6.51\times 10^{-3}$ & $1.21\times 10^{-2}$ & $1.93\times 10^{-2}$ \\ 
$  0.31$ & $5.57\times 10^{-4}$ & $2.47\times 10^{-3}$ & $5.99\times 10^{-3}$ & $1.12\times 10^{-2}$ & $1.79\times 10^{-2}$ \\ 
$  0.32$ & $5.08\times 10^{-4}$ & $2.26\times 10^{-3}$ & $5.52\times 10^{-3}$ & $1.03\times 10^{-2}$ & $1.66\times 10^{-2}$ \\ 
$  0.33$ & $4.65\times 10^{-4}$ & $2.08\times 10^{-3}$ & $5.10\times 10^{-3}$ & $9.57\times 10^{-3}$ & $1.55\times 10^{-2}$ \\ 
$  0.34$ & $4.27\times 10^{-4}$ & $1.92\times 10^{-3}$ & $4.71\times 10^{-3}$ & $8.88\times 10^{-3}$ & $1.44\times 10^{-2}$ \\ 
$  0.35$ & $3.92\times 10^{-4}$ & $1.77\times 10^{-3}$ & $4.36\times 10^{-3}$ & $8.25\times 10^{-3}$ & $1.34\times 10^{-2}$ \\ 
$  0.36$ & $3.61\times 10^{-4}$ & $1.64\times 10^{-3}$ & $4.05\times 10^{-3}$ & $7.68\times 10^{-3}$ & $1.25\times 10^{-2}$ \\ 
$  0.37$ & $3.33\times 10^{-4}$ & $1.51\times 10^{-3}$ & $3.76\times 10^{-3}$ & $7.15\times 10^{-3}$ & $1.17\times 10^{-2}$ \\ 
$  0.38$ & $3.08\times 10^{-4}$ & $1.40\times 10^{-3}$ & $3.50\times 10^{-3}$ & $6.67\times 10^{-3}$ & $1.09\times 10^{-2}$ \\ 
$  0.39$ & $2.85\times 10^{-4}$ & $1.30\times 10^{-3}$ & $3.26\times 10^{-3}$ & $6.23\times 10^{-3}$ & $1.02\times 10^{-2}$ \\ 
$  0.40$ & $2.64\times 10^{-4}$ & $1.21\times 10^{-3}$ & $3.04\times 10^{-3}$ & $5.82\times 10^{-3}$ & $9.59\times 10^{-3}$ \\ 
$  0.41$ & $2.45\times 10^{-4}$ & $1.13\times 10^{-3}$ & $2.83\times 10^{-3}$ & $5.45\times 10^{-3}$ & $9.00\times 10^{-3}$ \\ 
$  0.42$ & $2.27\times 10^{-4}$ & $1.05\times 10^{-3}$ & $2.65\times 10^{-3}$ & $5.10\times 10^{-3}$ & $8.44\times 10^{-3}$ \\ 
$  0.43$ & $2.12\times 10^{-4}$ & $9.81\times 10^{-4}$ & $2.48\times 10^{-3}$ & $4.78\times 10^{-3}$ & $7.93\times 10^{-3}$ \\ 
$  0.44$ & $1.97\times 10^{-4}$ & $9.16\times 10^{-4}$ & $2.32\times 10^{-3}$ & $4.48\times 10^{-3}$ & $7.45\times 10^{-3}$ \\ 
$  0.45$ & $1.84\times 10^{-4}$ & $8.56\times 10^{-4}$ & $2.17\times 10^{-3}$ & $4.21\times 10^{-3}$ & $7.01\times 10^{-3}$ \\ 
$  0.46$ & $1.71\times 10^{-4}$ & $8.01\times 10^{-4}$ & $2.03\times 10^{-3}$ & $3.95\times 10^{-3}$ & $6.60\times 10^{-3}$ \\ 
$  0.47$ & $1.60\times 10^{-4}$ & $7.50\times 10^{-4}$ & $1.91\times 10^{-3}$ & $3.72\times 10^{-3}$ & $6.21\times 10^{-3}$ \\ 
$  0.48$ & $1.50\times 10^{-4}$ & $7.02\times 10^{-4}$ & $1.79\times 10^{-3}$ & $3.50\times 10^{-3}$ & $5.85\times 10^{-3}$ \\ 
$  0.49$ & $1.40\times 10^{-4}$ & $6.58\times 10^{-4}$ & $1.68\times 10^{-3}$ & $3.29\times 10^{-3}$ & $5.52\times 10^{-3}$ \\ 
$  0.50$ & $1.31\times 10^{-4}$ & $6.18\times 10^{-4}$ & $1.58\times 10^{-3}$ & $3.10\times 10^{-3}$ & $5.20\times 10^{-3}$ \\ 
$  0.51$ & $1.23\times 10^{-4}$ & $5.80\times 10^{-4}$ & $1.49\times 10^{-3}$ & $2.92\times 10^{-3}$ & $4.91\times 10^{-3}$ \\ 
$  0.52$ & $1.15\times 10^{-4}$ & $5.45\times 10^{-4}$ & $1.40\times 10^{-3}$ & $2.75\times 10^{-3}$ & $4.63\times 10^{-3}$ \\ 
$  0.53$ & $1.08\times 10^{-4}$ & $5.12\times 10^{-4}$ & $1.32\times 10^{-3}$ & $2.59\times 10^{-3}$ & $4.37\times 10^{-3}$ \\ 
$  0.54$ & $1.01\times 10^{-4}$ & $4.81\times 10^{-4}$ & $1.24\times 10^{-3}$ & $2.44\times 10^{-3}$ & $4.13\times 10^{-3}$ \\ 
$  0.55$ & $9.51\times 10^{-5}$ & $4.52\times 10^{-4}$ & $1.17\times 10^{-3}$ & $2.30\times 10^{-3}$ & $3.90\times 10^{-3}$ \\ 
$  0.56$ & $8.93\times 10^{-5}$ & $4.26\times 10^{-4}$ & $1.10\times 10^{-3}$ & $2.17\times 10^{-3}$ & $3.68\times 10^{-3}$ \\ 
$  0.57$ & $8.39\times 10^{-5}$ & $4.01\times 10^{-4}$ & $1.04\times 10^{-3}$ & $2.05\times 10^{-3}$ & $3.48\times 10^{-3}$ \\ 
$  0.58$ & $7.89\times 10^{-5}$ & $3.77\times 10^{-4}$ & $9.78\times 10^{-4}$ & $1.94\times 10^{-3}$ & $3.29\times 10^{-3}$ \\ 
$  0.59$ & $7.41\times 10^{-5}$ & $3.55\times 10^{-4}$ & $9.22\times 10^{-4}$ & $1.83\times 10^{-3}$ & $2.99\times 10^{-3}$ \\ 
$  0.60$ & $6.97\times 10^{-5}$ & $3.34\times 10^{-4}$ & $8.69\times 10^{-4}$ & $1.64\times 10^{-3}$ & $2.61\times 10^{-3}$ \\ 
$  0.61$ & $6.56\times 10^{-5}$ & $3.15\times 10^{-4}$ & $7.99\times 10^{-4}$ & $1.45\times 10^{-3}$ & $2.30\times 10^{-3}$ \\ 
$  0.62$ & $6.16\times 10^{-5}$ & $2.96\times 10^{-4}$ & $7.12\times 10^{-4}$ & $1.29\times 10^{-3}$ & $2.05\times 10^{-3}$ \\ 
$  0.63$ & $5.80\times 10^{-5}$ & $2.78\times 10^{-4}$ & $6.38\times 10^{-4}$ & $1.16\times 10^{-3}$ & $1.84\times 10^{-3}$ \\ 
$  0.64$ & $5.45\times 10^{-5}$ & $2.51\times 10^{-4}$ & $5.76\times 10^{-4}$ & $1.04\times 10^{-3}$ & $1.66\times 10^{-3}$ \\ 
$  0.65$ & $5.13\times 10^{-5}$ & $2.27\times 10^{-4}$ & $5.22\times 10^{-4}$ & $9.46\times 10^{-4}$ & $1.51\times 10^{-3}$ \\ 
$  0.66$ & $4.82\times 10^{-5}$ & $2.07\times 10^{-4}$ & $4.75\times 10^{-4}$ & $8.61\times 10^{-4}$ & $1.37\times 10^{-3}$ \\ 
$  0.67$ & $4.53\times 10^{-5}$ & $1.89\times 10^{-4}$ & $4.33\times 10^{-4}$ & $7.86\times 10^{-4}$ & $1.25\times 10^{-3}$ \\ 
$  0.68$ & $4.23\times 10^{-5}$ & $1.73\times 10^{-4}$ & $3.97\times 10^{-4}$ & $7.20\times 10^{-4}$ & $1.15\times 10^{-3}$ \\ 
$  0.69$ & $3.88\times 10^{-5}$ & $1.58\times 10^{-4}$ & $3.64\times 10^{-4}$ & $6.60\times 10^{-4}$ & $1.05\times 10^{-3}$ \\ 
$  0.70$ & $3.56\times 10^{-5}$ & $1.45\times 10^{-4}$ & $3.35\times 10^{-4}$ & $6.07\times 10^{-4}$ & $9.68\times 10^{-4}$ \\ 
$  0.71$ & $3.28\times 10^{-5}$ & $1.34\times 10^{-4}$ & $3.08\times 10^{-4}$ & $5.59\times 10^{-4}$ & $8.91\times 10^{-4}$ \\ 
$  0.72$ & $3.02\times 10^{-5}$ & $1.23\times 10^{-4}$ & $2.84\times 10^{-4}$ & $5.15\times 10^{-4}$ & $8.21\times 10^{-4}$ \\ 
$  0.73$ & $2.79\times 10^{-5}$ & $1.14\times 10^{-4}$ & $2.62\times 10^{-4}$ & $4.76\times 10^{-4}$ & $7.58\times 10^{-4}$ \\ 
$  0.74$ & $2.57\times 10^{-5}$ & $1.05\times 10^{-4}$ & $2.42\times 10^{-4}$ & $4.39\times 10^{-4}$ & $7.00\times 10^{-4}$ \\ 
$  0.75$ & $2.37\times 10^{-5}$ & $9.70\times 10^{-5}$ & $2.23\times 10^{-4}$ & $4.06\times 10^{-4}$ & $6.47\times 10^{-4}$ \\ 
$  0.76$ & $2.19\times 10^{-5}$ & $8.95\times 10^{-5}$ & $2.06\times 10^{-4}$ & $3.75\times 10^{-4}$ & $5.97\times 10^{-4}$ \\ 
$  0.77$ & $2.02\times 10^{-5}$ & $8.27\times 10^{-5}$ & $1.90\times 10^{-4}$ & $3.46\times 10^{-4}$ & $5.52\times 10^{-4}$ \\ 
$  0.78$ & $1.87\times 10^{-5}$ & $7.63\times 10^{-5}$ & $1.76\times 10^{-4}$ & $3.20\times 10^{-4}$ & $5.10\times 10^{-4}$ \\ 
$  0.79$ & $1.72\times 10^{-5}$ & $7.04\times 10^{-5}$ & $1.62\times 10^{-4}$ & $2.95\times 10^{-4}$ & $4.70\times 10^{-4}$ \\ 
$  0.80$ & $1.59\times 10^{-5}$ & $6.48\times 10^{-5}$ & $1.50\times 10^{-4}$ & $2.72\times 10^{-4}$ & $4.34\times 10^{-4}$ \\ 
$  0.81$ & $1.46\times 10^{-5}$ & $5.97\times 10^{-5}$ & $1.38\times 10^{-4}$ & $2.51\times 10^{-4}$ & $3.99\times 10^{-4}$ \\ 
$  0.82$ & $1.34\times 10^{-5}$ & $5.48\times 10^{-5}$ & $1.27\times 10^{-4}$ & $2.30\times 10^{-4}$ & $3.67\times 10^{-4}$ \\ 
$  0.83$ & $1.23\times 10^{-5}$ & $5.03\times 10^{-5}$ & $1.16\times 10^{-4}$ & $2.12\times 10^{-4}$ & $3.37\times 10^{-4}$ \\ 
$  0.84$ & $1.13\times 10^{-5}$ & $4.60\times 10^{-5}$ & $1.06\times 10^{-4}$ & $1.94\times 10^{-4}$ & $3.09\times 10^{-4}$ \\ 
$  0.85$ & $1.03\times 10^{-5}$ & $4.20\times 10^{-5}$ & $9.71\times 10^{-5}$ & $1.77\times 10^{-4}$ & $2.82\times 10^{-4}$ \\ 
$  0.86$ & $9.35\times 10^{-6}$ & $3.82\times 10^{-5}$ & $8.84\times 10^{-5}$ & $1.61\times 10^{-4}$ & $2.57\times 10^{-4}$ \\ 
$  0.87$ & $8.47\times 10^{-6}$ & $3.46\times 10^{-5}$ & $8.02\times 10^{-5}$ & $1.46\times 10^{-4}$ & $2.33\times 10^{-4}$ \\ 
$  0.88$ & $7.64\times 10^{-6}$ & $3.12\times 10^{-5}$ & $7.24\times 10^{-5}$ & $1.32\times 10^{-4}$ & $2.11\times 10^{-4}$ \\ 
$  0.89$ & $6.85\times 10^{-6}$ & $2.80\times 10^{-5}$ & $6.50\times 10^{-5}$ & $1.19\times 10^{-4}$ & $1.89\times 10^{-4}$ \\ 
$  0.90$ & $6.10\times 10^{-6}$ & $2.49\times 10^{-5}$ & $5.80\times 10^{-5}$ & $1.06\times 10^{-4}$ & $1.69\times 10^{-4}$ \\ 
   \hline
\end{tabular}}
}
\end{table}

\begin{table}
\caption{\label{tab:r_conc2}Table of values of $\min\{D(\theta^2, 2\tau - 1, 50, -1/2), D(\theta, \tau, 100, -1/4)\}$ for $\theta \in \{0.06,0.07,0.08,0.09,0.1\}$.}
\centering
{\scriptsize
\fbox{%
\begin{tabular}{cccccc}
& \multicolumn{5}{c}{$\theta$} \\
\cline{2-6}
$\tau$ & $0.06$ & $0.07$ & $0.08$ & $0.09$ & $0.10$ \\ 
  \hline
$0.30$ & $2.81\times 10^{-2}$ & $3.82\times 10^{-2}$ & $4.97\times 10^{-2}$ & $6.24\times 10^{-2}$ & $7.63\times 10^{-2}$ \\ 
$  0.31$ & $2.61\times 10^{-2}$ & $3.57\times 10^{-2}$ & $4.64\times 10^{-2}$ & $5.84\times 10^{-2}$ & $7.14\times 10^{-2}$ \\ 
$  0.32$ & $2.43\times 10^{-2}$ & $3.33\times 10^{-2}$ & $4.35\times 10^{-2}$ & $5.47\times 10^{-2}$ & $6.70\times 10^{-2}$ \\ 
$  0.33$ & $2.27\times 10^{-2}$ & $3.12\times 10^{-2}$ & $4.08\times 10^{-2}$ & $5.14\times 10^{-2}$ & $6.30\times 10^{-2}$ \\ 
$  0.34$ & $2.12\times 10^{-2}$ & $2.92\times 10^{-2}$ & $3.83\times 10^{-2}$ & $4.83\times 10^{-2}$ & $5.93\times 10^{-2}$ \\ 
$  0.35$ & $1.98\times 10^{-2}$ & $2.73\times 10^{-2}$ & $3.59\times 10^{-2}$ & $4.55\times 10^{-2}$ & $5.59\times 10^{-2}$ \\ 
$  0.36$ & $1.85\times 10^{-2}$ & $2.57\times 10^{-2}$ & $3.38\times 10^{-2}$ & $4.29\times 10^{-2}$ & $5.28\times 10^{-2}$ \\ 
$  0.37$ & $1.74\times 10^{-2}$ & $2.41\times 10^{-2}$ & $3.18\times 10^{-2}$ & $4.04\times 10^{-2}$ & $4.99\times 10^{-2}$ \\ 
$  0.38$ & $1.63\times 10^{-2}$ & $2.26\times 10^{-2}$ & $2.99\times 10^{-2}$ & $3.81\times 10^{-2}$ & $4.72\times 10^{-2}$ \\ 
$  0.39$ & $1.53\times 10^{-2}$ & $2.13\times 10^{-2}$ & $2.82\times 10^{-2}$ & $3.60\times 10^{-2}$ & $4.46\times 10^{-2}$ \\ 
$  0.40$ & $1.43\times 10^{-2}$ & $2.00\times 10^{-2}$ & $2.66\times 10^{-2}$ & $3.40\times 10^{-2}$ & $4.22\times 10^{-2}$ \\ 
$  0.41$ & $1.35\times 10^{-2}$ & $1.89\times 10^{-2}$ & $2.51\times 10^{-2}$ & $3.22\times 10^{-2}$ & $4.00\times 10^{-2}$ \\ 
$  0.42$ & $1.27\times 10^{-2}$ & $1.78\times 10^{-2}$ & $2.37\times 10^{-2}$ & $3.04\times 10^{-2}$ & $3.79\times 10^{-2}$ \\ 
$  0.43$ & $1.19\times 10^{-2}$ & $1.68\times 10^{-2}$ & $2.24\times 10^{-2}$ & $2.88\times 10^{-2}$ & $3.59\times 10^{-2}$ \\ 
$  0.44$ & $1.12\times 10^{-2}$ & $1.58\times 10^{-2}$ & $2.11\times 10^{-2}$ & $2.72\times 10^{-2}$ & $3.40\times 10^{-2}$ \\ 
$  0.45$ & $1.06\times 10^{-2}$ & $1.49\times 10^{-2}$ & $2.00\times 10^{-2}$ & $2.58\times 10^{-2}$ & $3.23\times 10^{-2}$ \\ 
$  0.46$ & $9.98\times 10^{-3}$ & $1.41\times 10^{-2}$ & $1.89\times 10^{-2}$ & $2.44\times 10^{-2}$ & $3.06\times 10^{-2}$ \\ 
$  0.47$ & $9.41\times 10^{-3}$ & $1.33\times 10^{-2}$ & $1.79\times 10^{-2}$ & $2.31\times 10^{-2}$ & $2.90\times 10^{-2}$ \\ 
$  0.48$ & $8.88\times 10^{-3}$ & $1.26\times 10^{-2}$ & $1.69\times 10^{-2}$ & $2.19\times 10^{-2}$ & $2.76\times 10^{-2}$ \\ 
$  0.49$ & $8.38\times 10^{-3}$ & $1.19\times 10^{-2}$ & $1.60\times 10^{-2}$ & $2.08\times 10^{-2}$ & $2.62\times 10^{-2}$ \\ 
$  0.50$ & $7.92\times 10^{-3}$ & $1.12\times 10^{-2}$ & $1.52\times 10^{-2}$ & $1.97\times 10^{-2}$ & $2.48\times 10^{-2}$ \\ 
$  0.51$ & $7.48\times 10^{-3}$ & $1.06\times 10^{-2}$ & $1.44\times 10^{-2}$ & $1.87\times 10^{-2}$ & $2.36\times 10^{-2}$ \\ 
$  0.52$ & $7.07\times 10^{-3}$ & $1.01\times 10^{-2}$ & $1.36\times 10^{-2}$ & $1.77\times 10^{-2}$ & $2.24\times 10^{-2}$ \\ 
$  0.53$ & $6.68\times 10^{-3}$ & $9.53\times 10^{-3}$ & $1.29\times 10^{-2}$ & $1.68\times 10^{-2}$ & $2.13\times 10^{-2}$ \\ 
$  0.54$ & $6.32\times 10^{-3}$ & $9.02\times 10^{-3}$ & $1.22\times 10^{-2}$ & $1.60\times 10^{-2}$ & $2.02\times 10^{-2}$ \\ 
$  0.55$ & $5.98\times 10^{-3}$ & $8.54\times 10^{-3}$ & $1.16\times 10^{-2}$ & $1.52\times 10^{-2}$ & $1.92\times 10^{-2}$ \\ 
$  0.56$ & $5.65\times 10^{-3}$ & $8.09\times 10^{-3}$ & $1.10\times 10^{-2}$ & $1.44\times 10^{-2}$ & $1.83\times 10^{-2}$ \\ 
$  0.57$ & $5.35\times 10^{-3}$ & $7.66\times 10^{-3}$ & $1.04\times 10^{-2}$ & $1.37\times 10^{-2}$ & $1.73\times 10^{-2}$ \\ 
$  0.58$ & $5.06\times 10^{-3}$ & $7.13\times 10^{-3}$ & $9.49\times 10^{-3}$ & $1.22\times 10^{-2}$ & $1.54\times 10^{-2}$ \\ 
$  0.59$ & $4.39\times 10^{-3}$ & $6.09\times 10^{-3}$ & $8.10\times 10^{-3}$ & $1.04\times 10^{-2}$ & $1.31\times 10^{-2}$ \\ 
$  0.60$ & $3.82\times 10^{-3}$ & $5.30\times 10^{-3}$ & $7.04\times 10^{-3}$ & $9.08\times 10^{-3}$ & $1.14\times 10^{-2}$ \\ 
$  0.61$ & $3.37\times 10^{-3}$ & $4.67\times 10^{-3}$ & $6.21\times 10^{-3}$ & $8.00\times 10^{-3}$ & $1.01\times 10^{-2}$ \\ 
$  0.62$ & $3.01\times 10^{-3}$ & $4.17\times 10^{-3}$ & $5.54\times 10^{-3}$ & $7.14\times 10^{-3}$ & $8.97\times 10^{-3}$ \\ 
$  0.63$ & $2.70\times 10^{-3}$ & $3.74\times 10^{-3}$ & $4.98\times 10^{-3}$ & $6.42\times 10^{-3}$ & $8.06\times 10^{-3}$ \\ 
$  0.64$ & $2.44\times 10^{-3}$ & $3.38\times 10^{-3}$ & $4.50\times 10^{-3}$ & $5.80\times 10^{-3}$ & $7.29\times 10^{-3}$ \\ 
$  0.65$ & $2.21\times 10^{-3}$ & $3.07\times 10^{-3}$ & $4.08\times 10^{-3}$ & $5.26\times 10^{-3}$ & $6.62\times 10^{-3}$ \\ 
$  0.66$ & $2.01\times 10^{-3}$ & $2.79\times 10^{-3}$ & $3.72\times 10^{-3}$ & $4.79\times 10^{-3}$ & $6.03\times 10^{-3}$ \\ 
$  0.67$ & $1.84\times 10^{-3}$ & $2.55\times 10^{-3}$ & $3.40\times 10^{-3}$ & $4.38\times 10^{-3}$ & $5.51\times 10^{-3}$ \\ 
$  0.68$ & $1.68\times 10^{-3}$ & $2.34\times 10^{-3}$ & $3.11\times 10^{-3}$ & $4.01\times 10^{-3}$ & $5.05\times 10^{-3}$ \\ 
$  0.69$ & $1.55\times 10^{-3}$ & $2.14\times 10^{-3}$ & $2.86\times 10^{-3}$ & $3.68\times 10^{-3}$ & $4.64\times 10^{-3}$ \\ 
$  0.70$ & $1.42\times 10^{-3}$ & $1.97\times 10^{-3}$ & $2.63\times 10^{-3}$ & $3.39\times 10^{-3}$ & $4.27\times 10^{-3}$ \\ 
$  0.71$ & $1.31\times 10^{-3}$ & $1.82\times 10^{-3}$ & $2.42\times 10^{-3}$ & $3.12\times 10^{-3}$ & $3.93\times 10^{-3}$ \\ 
$  0.72$ & $1.21\times 10^{-3}$ & $1.68\times 10^{-3}$ & $2.23\times 10^{-3}$ & $2.88\times 10^{-3}$ & $3.63\times 10^{-3}$ \\ 
$  0.73$ & $1.11\times 10^{-3}$ & $1.55\times 10^{-3}$ & $2.06\times 10^{-3}$ & $2.66\times 10^{-3}$ & $3.35\times 10^{-3}$ \\ 
$  0.74$ & $1.03\times 10^{-3}$ & $1.43\times 10^{-3}$ & $1.90\times 10^{-3}$ & $2.46\times 10^{-3}$ & $3.09\times 10^{-3}$ \\ 
$  0.75$ & $9.51\times 10^{-4}$ & $1.32\times 10^{-3}$ & $1.76\times 10^{-3}$ & $2.27\times 10^{-3}$ & $2.86\times 10^{-3}$ \\ 
$  0.76$ & $8.78\times 10^{-4}$ & $1.22\times 10^{-3}$ & $1.63\times 10^{-3}$ & $2.10\times 10^{-3}$ & $2.64\times 10^{-3}$ \\ 
$  0.77$ & $8.12\times 10^{-4}$ & $1.13\times 10^{-3}$ & $1.50\times 10^{-3}$ & $1.94\times 10^{-3}$ & $2.44\times 10^{-3}$ \\ 
$  0.78$ & $7.50\times 10^{-4}$ & $1.04\times 10^{-3}$ & $1.39\times 10^{-3}$ & $1.79\times 10^{-3}$ & $2.26\times 10^{-3}$ \\ 
$  0.79$ & $6.92\times 10^{-4}$ & $9.61\times 10^{-4}$ & $1.28\times 10^{-3}$ & $1.65\times 10^{-3}$ & $2.08\times 10^{-3}$ \\ 
$  0.80$ & $6.38\times 10^{-4}$ & $8.86\times 10^{-4}$ & $1.18\times 10^{-3}$ & $1.53\times 10^{-3}$ & $1.92\times 10^{-3}$ \\ 
$  0.81$ & $5.88\times 10^{-4}$ & $8.16\times 10^{-4}$ & $1.09\times 10^{-3}$ & $1.41\times 10^{-3}$ & $1.77\times 10^{-3}$ \\ 
$  0.82$ & $5.41\times 10^{-4}$ & $7.51\times 10^{-4}$ & $1.00\times 10^{-3}$ & $1.29\times 10^{-3}$ & $1.63\times 10^{-3}$ \\ 
$  0.83$ & $4.97\times 10^{-4}$ & $6.89\times 10^{-4}$ & $9.20\times 10^{-4}$ & $1.19\times 10^{-3}$ & $1.50\times 10^{-3}$ \\ 
$  0.84$ & $4.55\times 10^{-4}$ & $6.32\times 10^{-4}$ & $8.43\times 10^{-4}$ & $1.09\times 10^{-3}$ & $1.37\times 10^{-3}$ \\ 
$  0.85$ & $4.16\times 10^{-4}$ & $5.77\times 10^{-4}$ & $7.71\times 10^{-4}$ & $9.95\times 10^{-4}$ & $1.25\times 10^{-3}$ \\ 
$  0.86$ & $3.79\times 10^{-4}$ & $5.26\times 10^{-4}$ & $7.02\times 10^{-4}$ & $9.07\times 10^{-4}$ & $1.14\times 10^{-3}$ \\ 
$  0.87$ & $3.44\times 10^{-4}$ & $4.77\times 10^{-4}$ & $6.37\times 10^{-4}$ & $8.23\times 10^{-4}$ & $1.04\times 10^{-3}$ \\ 
$  0.88$ & $3.11\times 10^{-4}$ & $4.31\times 10^{-4}$ & $5.76\times 10^{-4}$ & $7.44\times 10^{-4}$ & $9.37\times 10^{-4}$ \\ 
$  0.89$ & $2.79\times 10^{-4}$ & $3.88\times 10^{-4}$ & $5.18\times 10^{-4}$ & $6.69\times 10^{-4}$ & $8.42\times 10^{-4}$ \\ 
$  0.90$ & $2.49\times 10^{-4}$ & $3.46\times 10^{-4}$ & $4.63\times 10^{-4}$ & $5.97\times 10^{-4}$ & $7.53\times 10^{-4}$ \\
   \hline
\end{tabular}}
}
\end{table}


\end{document}